\documentclass[10pt]{article}

\usepackage[colorlinks,linkcolor=blue,citecolor=blue]{hyperref}
\usepackage{amsthm}
\usepackage{multirow}
\usepackage{marginnote}
\usepackage{a4wide}
\usepackage{amssymb}
\usepackage{amsfonts}
\usepackage{amsmath}
\usepackage{mathrsfs}
\usepackage{tikz}
\usetikzlibrary{arrows,matrix}
\usetikzlibrary{positioning}
\usepackage{mdframed} % draw fram of texts
\usepackage{lipsum} % this is used for drawing frams of texts
\usepackage{extarrows} % making long equal sign (or arrows) with text under or above
\usepackage{color}
\usepackage{enumitem} 
\usepackage{tocloft} % change the style of table of contents
\usepackage{titlesec} % control the space before and after section titles

\usepackage[all]{xy} % draw diagrams

\input xy
\xyoption{arrow} \xyoption{matrix}

\date{}

\newtheorem{proposition}{Proposition}[section]
\newtheorem{theorem}[proposition]{Theorem}
\newtheorem{lemma}[proposition]{Lemma}
\newtheorem{example}[proposition]{Example}
\newtheorem{definition}[proposition]{Definition}
\newtheorem{corollary}[proposition]{Corollary}

\def\GK{{\rm  GK}\,}

\def\Hom{{\rm Hom}}
\def\der{\partial }

\def\nFM0{{\nu }_{F,M_0}}
\def\nFN0{{\nu }_{F,N_0}}
\def\nGN0{{\nu }_{G,N_0}}

\def\N0{ {\bf N}_0 }

\def\t{\otimes}
\def\g{\gamma}
\def\v{\varphi}
\def\ra{\rightarrow}

\def\Xpm{X^{\pm }}

\def\s{\sigma}
\def\Z{\mathbb{Z}}

\def\l1{{\lambda}_1}

\def\a{\alpha}
\def\a0{ {\alpha }_0}
\def\a1{ {\alpha }_1}

\def\l{\lambda}

\def\nFGM0{{\nu }_{F,G,M_0}}

\def\nFN0{{\nu}_{F,N_0}}

\def\sm{{\sigma}^m}

\def\sm1{{\sigma}^{-1}}

\def\smtp1{{\sigma}^{-t+1}}

\def\S1{S^{-1}}

\def\Xpm1{X^{\pm 1}_1}

\def\sPM1{{\sigma }^{\pm 1}}
\def\sMP1{{\sigma }^{\mp 1 }}

\def\b{\beta}
\def\d{\delta}

\def\di{{\rm d.ind}}

\def\L{\Lambda}

\def\CA{{\cal A}}

\def\Ytm1{Y^{t-1}}
\def\Yim1{Y^{i-1}}

\def\CL{{\cal L}}

\def\CS{{\cal S}}
\def\CF{{\cal F}}

\def\CH{{\cal H}}
\def\ass{{\rm ass}}

\def\Aut{{\rm Aut}}

\def\dim{{\rm dim }}

\def\ker{ {\rm ker } }

\def\D{ \Delta }
\def\SL2Z{ {\rm SL}_2({\bf Z}) }

\def\th{ \theta }

\def\CL{{\cal L}}

\def\Gp1{ G^{1 , 1 } }
\def\P11{ P^{-1 , 1 } }
\def\Pp1{ P^{1 , 1 } }

\def\th{\theta}

\def\nCLsr{{}^\nu\kern-2pt {\cal L}^{\sigma , \rho  }}
\def\nP{{}^\nu \kern-2pt P}
\def\nL{{}^\nu\kern-2pt L}
\def\nLL{{}^\nu\kern-2pt \Lambda}
\def\nPsr{{}^\nu\kern-2pt P^{\sigma , \rho  }}
\def\nLsr{{}^\nu\kern-2pt L^{\sigma , \rho  }}
\def\nuCL{{}^\nu\kern-2pt  {\cal L}}
\def\nCLsr{{}^\nu\kern-2pt {\cal L}^{\sigma , \rho  }}
\def\nCL1m{{}^\nu\kern-2pt {\cal L}^{-1 , 1  }}
\def\x1nu{x^\frac{1}{\nu}}
\def\xm1nu{x^{-\frac{1}{\nu}}}

\def\ra{\rightarrow }

\def\CB{{\cal B}}

\def\CI{{\cal I}}

\def\CT{{\cal T}}

\def\CC{ {\cal C}}

\def\CH{ {\cal H}}
\def\CP{ {\cal P}}

\def\nAM0{{\nu }_{{\cal A},M_0}}
\def\nAN0{{\nu }_{{\cal A},N_0}}

\def\End{ {\rm End }}

\def\CP{ {\cal P }}

\def\bR{\overline{R}}

\def\ga{\mathfrak{a}}
\def\gb{\mathfrak{b}}

\def\gp{\mathfrak{p}}

\def\SL{{\rm SL}}

\def\Hom{{\rm Hom}}

\def\di!{\frac{\der^i}{i!}}
\def\dik!{\frac{\der^k_i}{k!}}
\def\gl{\mathfrak{l}}

\def\id{{\rm id}}

\def\N{\mathbb{N}}

\def\0{\overline{0}}
\def\1{\overline{1}}

\def\Ln1{\L_{n,\overline{1}}}

\def\a1{a_{\overline{1}}}

\def\bs{\overline{s}}

\def\S{\Sigma}

\def\vn1{\overrightarrow{n-1}}

\def\gl{{\rm gl}}
\def\sl{{\rm sl}}

\def\mA{\mathbb{A}}
\def\mD{\mathbb{D}}

\def\mS{\mathbb{S}}
\def\mJ{\mathbb{J}}
\def\mI{\mathbb{I}}

\def\lann{{\rm l.ann}}
\def\rann{{\rm r.ann}}

\def\clKdim{{\rm cl.Kdim}}

\def\mE{\mathbb{E}}

\def\K1{{\rm K}_1}

\def\mY{\mathbb{Y}}

\def\hmI1{\widehat{\mI_1}}
\def\tmI1{\widetilde{\mI_1}}
\def\tmJ1{\widetilde{\mJ_1}}
\def\hB1{\widehat{B_1}}
\def\hCB1{\widehat{\CB_1}}

\def\bS{\overline{S}}

\def\Den{{\rm Den}}

\def\Ore{{\rm Ore}}

\def\Den{{\rm Den}}

\def\Ass{{\rm Ass}}

\def\br{\overline{r}}

\def\bs{\overline{s}}

\def\ga{\mathfrak{a}}

\def\udim{{\rm udim}}

\def \S{\mathcal{S}}

\def\sl2{\mathfrak{sl}_2}

\def\sl2{\mathfrak{sl}_2}
\def\gl2{\mathfrak{gl}_2}

\def\b1{\overline{1}}

\def\RR{\mathbb{R}}
\def\Z{\mathbb{Z}}

\def\gl{{\mathfrak{l}}}

\def\Den{{\rm Den}}

\def\Ore{{\rm Ore}}

\def\Den{{\rm Den}}

\def\Ass{{\rm Ass}}

\def\br{\overline{r}}

\def\bs{\overline{s}}

\def\ga{\mathfrak{a}}

\def\udim{{\rm udim}}

\def\pCC{{}'\CC}
\def\pCCR{{}'\CC_R}

\def\RR{{}_RR}
\def\pQ{{}'Q}

\def\pbCC{{}'\overline{\CC}}

\def\pSlR{{}'S_l(R)}
\def\pQlR{{}'Q_l(R)}
\def\pS{{}'S}
\def\pga{{}'\ga}

\def\RSm{ R\langle S^{-1}\rangle}

\def\pCCR{{}'\CC_R}

\def\RR{\mathbb{R}}
\def\Z{\mathbb{Z}}

\def\CS{{\cal S}}
\def\CL{{\cal L}}

\def \mrL{\mathrm{L}} 
\def\mrLR{ \mathrm{L}(R)}
\def\mrLRa{ \mathrm{L}(R, \ga )}

\makeatletter
\newenvironment{proof*}[1][\proofname]{\par
  \pushQED{\qed}%
  \normalfont \partopsep=\z@skip \topsep=\z@skip
  \trivlist
  \item[\hskip\labelsep
        \itshape
    #1\@addpunct{.}]\ignorespaces
}{%
  \popQED\endtrivlist\@endpefalse
}
\makeatother

\begin{document}

\author{V. V. \  Bavula  %(OreDen-lregquot-ring.tex)
}

\title{Ore sets, denominator sets and the  left regular left quotient ring of a ring}

\maketitle

\begin{abstract}

The aim of the papers is to describe the left regular left quotient ring ${}'Q(R)$ and the right  regular right quotient ring $Q'(R)$ for the following algebras $R$: $\mS_n=\mS_1^{\t n}$ is the algebra of one-sided inverses, where $\mS_1=K\langle x,y\, | \, yx=1\rangle$, $\CI_n=K\langle \der_1, \ldots, \der_n,\int_1,\ldots, \int_n\rangle$ is the algebra of scalar integro-differential operators and the Jacobian algebra $\mA_1=K\langle x,\der, (\der x)^{-1}\rangle$.  The sets of left and right regular elements of the algebras $\mS_1$, $\CI_1$,  $\mA_1$ and  $\mI_1=K\langle x, \der,\int\rangle$.   A progress is made on the following conjecture, \cite{Clas-lreg-quot}: $${}'Q(\mI_n)\simeq Q(A_n)\;\; {\rm 
 where}\;\;  \mI_n =K\bigg\langle x_1,\ldots , x_n,  \der_1, \ldots, \der_n,\int_1,\ldots, \int_n\bigg\rangle$$ is the algebra of polynomial  integro-differential operators and  $Q(A_n)$ is the classical quotient ring (of fractions) of the $n$'th Weyl algebra $A_n$, i.e. a criterion is given when the isomorphism holds. We produce several general constructions of left Ore and left denominator sets that appear naturally in applications and are of independent interest and use them to produce explicit left denominator sets that give the localization ring isomorphic to ${}'Q(\mS_n)$ or  ${}'Q(\mI_n)$ or ${}'Q(\mA_n)$ where $\mA_n:=\mA_1^{\t n}$. Several characterizations of one-sided regular elements of a ring are given in module-theoretic and one-sided-ideal-theoretic way.

$\noindent$



{ \small \tableofcontents}
\end{abstract}

%%%%%%%%%%%%%%%%%% SECTION 1 %%%%%%%%%%%%%%%%%%%%%

\section{Introduction} \label{INTR} %\marginpar{INTR}

In this paper, module means a {\em left} module.  The following
notation will remain fixed throughout the paper (if it is not
stated otherwise): 
\begin{itemize}

\item $R$ is a unital ring and $R^\times$ be its group of units, %$\gn =\gn_R$ is its prime radical and $\Min (R)$ is the set of minimal primes of $R$;

\item   $\CC = \CC_R$  is the set of {\em regular} elements of the ring $R$ (i.e.\ $\CC$ is the set of non-zero-divisors of the ring $R$); 

\item   $\pCCR $  is the set of {\em left  regular} elements of the ring $R$, i.e.\ $\pCCR :=\{ c\in R\, | \, \ker (\cdot c)=0\}$ where $\cdot c: R\ra R$, $r\mapsto rc$; 

\item $\CC_R':=\{ c\in R\, | \, \ker (c\cdot )=0\}$  is the set of {\em right regular} elements of  $R$;

\item   $Q=Q_{l,cl}(R):= \CC_R^{-1}R$ is the {\em left quotient ring}  (the {\em classical left ring of fractions}) of the ring $R$ (if it exists, i.e.\ if $\CC_R$ is a left Ore set) and $Q^\times$ is the group of units of $Q$;
%\item   $\gn =\gn_R$ is the  prime radical of $R$ and $\nu\in \N \cup \{ \infty \}$ is its {\em nilpotency degree} ($\gn^\nu \neq 0$ but $\gn^{\nu +1}=0$);
%\item   $\bR := R/ \gn$ and $\pi: R\ra \bR$, $r\mapsto \br =r+\gn$;
%\item   $\OCC := \CC_{\bR}$ is the set of regular elements of the ring $\bR$ and $\bQ := \OCC^{-1}\bR$ is its left quotient ring;
%\item   $\CC':= \pi^{-1}(\OCC):=\{ c\in R\, \, | \, c+\gn \in \OCC\}$ and $Q':=\CC'^{-1}R$ (if it exists),

\item $\Ore_l(R):=\{ S\, | \, S$ is a left Ore set in $R\}$; \item
$\Den_l(R):=\{ S\, | \, S$ is a left denominator set in $R\}$;

\item
$\Ass_l(R):= \{ \ass (S)\, | \, S\in \Den_l(R)\}$ where $\ass
(S):= \{ r\in R \, | \, sr=0$ for some $s=s(r)\in S\}$;

\item $\Den_l(R, \ga )$ is the set of left denominator sets $S$ of $R$ with $\ass (S)=\ga$ where $\ga$ is an ideal of $R$;
     %and $\ass (S):= \{r\in R\, | \, sr=0$ for some $s\in S\}$,
     
\item $S_\ga=S_\ga (R)=S_{l,\ga }(R)$
 is the {\em largest element} of the poset $(\Den_l(R, \ga ),
\subseteq )$ and $Q_\ga (R):=Q_{l,\ga }(R):=S_\ga^{-1} R$ is  the
{\em largest left quotient ring associated with} $\ga$. The fact that $S_\ga $
exists is proven in {\cite[Theorem 2.1]{larglquot}};

\item In particular, $S_0=S_0(R)=S_{l,0}(R)$ is the largest
element of the poset $(\Den_l(R, 0), \subseteq )$, i.e.\ the {\em largest regular  left Ore set} of $R$,  and
$Q_l(R):=S_0^{-1}R$ is the {\em largest left quotient ring} of $R$ \cite{larglquot};

%\item $\maxDen_l(R)$ is the set of maximal left denominator sets of $R$ (it is always a {\em non-empty} set, see \cite{larglquot}.
%\item the length of an $R$-module $M$ is denoted by $l_R(M)$;

\item $\pS (R)=\pSlR$ is the {\em largest left denominator set} in $\pCCR$ and $\pQ (R):=\pQ_l(R):= \pSlR^{-1}R$ is the {\em left regular left quotient ring} of $R$;

\item $\pga := \ass_R(\pSlR )$ and ${}'\pi : R\ra \bR':= R/ \pga$, $r\mapsto \br := r+\pga$;
            
\item $S'(R)=S_r'(R)$ is the {\em largest right denominator set} in $\CC_R'$ and $Q'(R):= Q_r'(R):= RS_r'(R)^{-1}$ is the {\em right  regular right quotient ring} of $R$.           
            
\end{itemize}

{\bf Semisimplicity criteria for the ring $\pQ_{l, cl} (R)$}.  For each element $r\in R$, let $r\cdot : R\ra R$, $x\mapsto rx$ and $\cdot r : R\ra R$, $x\mapsto xr$. The sets $\pCCR := \{ r\in R\, | \, \ker (\cdot r)=0\}$ and $\CC_R' := \{ r\in R\, | \, \ker (r\cdot )=0\}$ are called the {\em sets of left and right regular elements} of $R$, respectively.  Their intersection $\CC_R=\pCCR \cap \CC_R'$ is the {\em set of regular elements} of $R$. The rings $Q_{l,cl}(R):= \CC_R^{-1}R$ and $Q_{r,cl}(R):= R\CC_R^{-1}$ are called the {\em classical left and right quotient rings} of $R$, respectively. Goldie's Theorem states that the ring $Q_{l, cl}(R)$ is  a semisimple Artinian ring iff the ring $R$ is  semiprime, $\udim (\RR)<\infty$ and the ring $R$ satisfies the a.c.c. on left annihilators ($\udim$ stands for the uniform dimension). In \cite{Crit-S-Simp-lQuot}, four more new criteria are given based on different ideas, \cite[Theorems 3.1, 4.1, 5.1, 6.2]{Crit-S-Simp-lQuot}. 

In \cite{Clas-lreg-quot},   the rings $\pQ_{l, cl} (R) := \pCCR^{-1}R$ (the {\em classical left regular left quotient  ring} of $R$) and  $Q_{r, cl}' (R) :=R {\CC_R'}^{-1}$ (the {\em classical right regular right quotient  ring} of $R$) are introduced and studied,  and  several semisimplicity criteria for them are given  (\cite[Theorems 1.1, 3.1, 3.3, 3.4, 3.5]{Clas-lreg-quot}).  The ring $\pQ_{l,cl}(R)$ is a semisimple Artinian ring iff the ring $\pQ_l(R)$
is so (\cite[Theorem 4.3]{Clas-lreg-quot}).

A subset $S$ of a ring $R$ is called a {\em multiplicative set} if $1\in S$, $SS\subseteq S$ and $0\not\in S$. Suppose that $S$ and $T$ are multiplicative sets in $R$ such that $S\subseteq T$.   The multiplicative subset $S$ of $T$ is called {\em dense} (or {\em left dense}) in $T$ if for each element $t\in T$ there exists an element $r\in R$ such that $rt\in S$.
 For  a left ideal $I$ of $R$, let 
 $$
 \pCC_I:= \{ i\in I\, | \, \cdot i : I\ra I, \;\; x\mapsto xi\;\; {\rm  is\; an\; injection} \}.
 $$
   For a nonempty subset $S$ of a ring $R$, let $\ass_R(S) := \{ r\in R\, | \, sr=0$  for some $s\in S\}$. Let us mention a  semisimplicity criteria for the ring $\pQ_{l, cl}(R)$ that is used in the paper (see the proofs of  Theorem \ref{18Mar24} and Theorem \ref{25Mar24}).

\begin{theorem}\label{28Feb15}%\marginpar{28Feb15}
(\cite[Theorems 1.1]{Clas-lreg-quot})  Let $R$ be a ring, $\pCC = \pCC_R$ and $\ga := \ass_R(\pCC )$. The following statements are equivalent.
\begin{enumerate}
\item $\pQ:=\pQ_{l,cl}(R)$ is a semisimple Artinian ring.
\item
\begin{enumerate}
\item $\ga$ is a semiprime ideal of $R$,
\item the set $\pbCC:= \pi (\pCC )$ is a dense subset of $\pCC_{\bR}$
 where $\pi : R\ra \bR := R/ \ga$, $r\mapsto \br := r+\ga$,
\item $\udim ({}_{\bR}\bR ) <\infty$, and
\item $\pCC_V\neq \emptyset$ for all uniform left ideals $V$ of $\bR$.
\end{enumerate}
\item $\ga $ is a semiprime ideal of $R$, $\pbCC$ is a dense subset of $\CC_{\bR}$ and $Q_{l, cl}(\bR )$ is a semisimple Artinian ring.
\end{enumerate}

If one of the equivalent conditions holds then $\pbCC \in \Den_l( \bR , 0)$, $\pbCC$ is a dense subset of $\CC_{\bR }$ and $\pQ \simeq \pbCC^{-1} \bR \simeq Q_{l, cl} (\bR )$. Furthermore, the ring $\pQ$ is a simple ring iff the ideal $\ga$ is a prime ideal.
\end{theorem}

{\bf The left regular left quotient ring $\pQ_l(R)$ of a ring $R$ and its semisimplicity criteria}. Let $R$ be a ring. In general, the classical left quotient ring $Q_{l,cl}(R)$ does not exists, i.e. the set of regular elements $\CC_R$ of $R$ is not a left Ore set. The set $\CC_R$ contains the {\em largest left Ore set} denoted by $S_l(R)$ and the ring $Q_l(R) := S_l(R)^{-1}R$ is called the {\em (largest) left quotient ring} of $R$, \cite{larglquot}. Clearly, if $\CC_R$ is a left Ore set then $\CC_R = S_l(R)$ and $Q_{l,cl}(R) = Q_l(R)$. Similarly, the set $\pCCR$ of left regular elements of the ring $R$ is not a left denominator set, in general, and so in this case the classical left regular left quotient ring $\pQ_{l,cl}(R)$ does not exist. The set $\pCCR$ contains the {\em largest} left denominator set $\pSlR$ (\cite[Lemma 4.1.(1)]{Clas-lreg-quot}) and the ring $\pQ_l(R):= \pSlR^{-1} R$ is called the {\em left regular left quotient ring} of $R$, \cite{Clas-lreg-quot}. If $\pCCR$ is a left denominator set then $\pCCR = \pSlR$ and $\pQ_{l,cl}(R) = \pQ_l(R)$. 

The main difficulty  in constructing  the rings $\pQ(R)$ and $Q'(R)$ is to find descriptions of the sets $\pS (R)$ and $S'(R)$. The main idea in constructing the rings $\pQ(R)$ and $Q'(R)$ is to find larger or smaller or other denominator sets that give the {\em same} localization as the sets $\pS (R)$ and $S'(R)$ do. In order to do so, we produce several constructions of Ore or denominator sets (that satisfy various conditions, appear naturally in applications and are of  independent interest) and use them in the paper.

The paper is organized as follows. In Section \ref{ORE-DEN}, we present several results on and constructions of left Ore and denominator sets of a ring (Proposition \ref{A10Mar24}) and  give a sufficient condition  for the sets  ${}'\CC_R^{le}$, ${\CC_R'}^{re}$ and  $\CC_R^e$  being  denominator  sets (Proposition \ref{A11Mar24}). 
 Lemma \ref{a11Mar24}  and Corollary \ref{a10Mar24}  give  equivalent conditions to the left Ore condition. Lemma \ref{a15Mar24} makes connection between the sets of right or left regular elements of a ring and sets of module monomorphisms. Lemma \ref{b15Mar24} is an application of Lemma \ref{a15Mar24} for  one-sided ideals. For a module $M$ and its submodule $N$, Lemma \ref{c15Mar24} makes connections between the sets $\pCC_M$, $\CC_M'$ and $\CC_M$ and $\pCC_N$, $\CC_N'$ and $\CC_N$, respectively.  Corollary \ref{d15Mar24} and Corollary \ref{e15Mar24} are applications of the above result to one-sided essential ideals. These two corollaries are used in proofs.

 In Section \ref{RING-PQmSn}, the rings $\pQ (\mS_n)$ and $Q'(\mS_n)$ are described (Theorem \ref{19Mar24} and Corollary \ref{a22Mar24}.(1)) where 
 $\mS_n:=\mS_1^{\t n}$ is the algebra of one-sided inverses and $\mS_1:=K\langle x,y\, | \, yx=1\rangle$.  It is proven that $\pCC_{\mS_n}=\pS (\mS_n)$ and ${}'Q_{l,cl}(\mS_n)={}'Q(\mS_n)$ 
  (Corollary \ref{a31Mar24}), 
 $\CC_{\mS_n}'=S' (\mS_n)$ and $Q_{l,cl}'(\mS_n)=Q'(\mS_n)$ (Corollary {\ref{a22Mar24}.(2)). The algebra
$\mS_n$ is a non-commutative, non-Noetherian,  central, prime, catenary algebra; its ideals commute and 
satisfy the ascending chain condition;  its classical Krull dimension  is $2n$ but the
weak and the global dimensions  are $n$, \cite{shrekalg}. The same results hold for  the algebra of scalar integro-differential operators ($K$ is a field of characteristic zero),
 $$
\CI_n:= K\bigg\langle
 \der_1, \ldots ,\der_n,  \int_1,
\ldots , \int_n\bigg\rangle ,
$$ 
see Corollary \ref{a21Mar24}.  

Let $K$ be a  field of characteristic zero. The algebra $A_n=K\langle x_1, \ldots , x_n, \der_1, \ldots , \der_n\rangle $ is called the $n$'th  Weyl algebra. It is canonically isomorphic to the algebra of polynomial differential operators ($\der_i=\frac{\der}{\der x_i}$). The algebra $A_n$ is a Noetheian domain. Hence, by Goldie's Theorem it (left and right) classical quotient ring $Q(A_n)$ is a division ring.

The algebra 
$$\mI_n:=K\bigg\langle x_1, \ldots , x_n,
 \der_1, \ldots ,\der_n,  \int_1,
\ldots , \int_n\bigg\rangle $$
is called the algebra of polynomial integro-differential operators.  The algebra $\mI_n$ is a prime, central, catenary, non-Noetherian
algebra of classical Krull dimension $n$ and of Gelfand-Kirillov
dimension $2n$, \cite{algintdif}. In \cite{Clas-lreg-quot}, 
 explicit descriptions of the sets $\pCC_{\mI_1}$ and $\CC_{\mI_1}'$ are given (\cite[ Theorem 6.7]{Clas-lreg-quot}). These descriptions  are far from being trivial or obvious. It is also  proven  that 
$$ \pQ_{l,cl}(\mI_1)\simeq Q(A_1),
$$ 
see \cite[ Theorem 6.5.(1)]{Clas-lreg-quot}. In \cite{Clas-lreg-quot}, it is conjectured that 
$$ \pQ_{l,cl}(\mI_n)\simeq Q(A_n).
$$

In Section \ref{PQmIn-PQCIn}, we make progress on the conjecture. Namely,  
Theorem \ref{A25Mar24} is a criterion  for the the ring $\pQ (\mI_n)$ being isomorphic to quotient ring $Q(A_n)$. Despite the fact that there are no   descriptions yet for the set $\pCC_{\mI_n}$ and $\CC_{\mI_n}'$ where $n\geq 2$, Theorem \ref{18Mar24} provides explicit left denominators sets   $S\in \Den_l(\mI_n, \ga_n)$ such that  $S^{-1}\mI_n\simeq \pQ (\mI_n)$. 

\begin{definition} (\cite{jacalg}) The {\bf Jacobian algebra}
$\mA_n$ is the subalgebra of $\End_K(P_n)$ generated by the Weyl
algebra $A_n$ and the elements $H_1^{-1}, \ldots , H_n^{-1}\in
\End_K(P_n)$ where $$H_1:= \der_1x_1, \ldots , H_n:= \der_nx_n.$$
\end{definition}
The algebra $\mI_n$ properly  contains the algebras $A_n$, $\mI_n$ and $\CI_n$. 

In Section \ref{RI-PQAN}, a criterion is given  for  $\pQ (\mA_n)\simeq Q(A_n)$ (Theorem \ref{25Mar24}). As a corollary it is shown  that 
$$\pQ (\mA_1)\simeq Q(A_1)\;\; {\rm and}\;\; Q' (\mA_1)\simeq Q(A_1),$$ 
see Theorem \ref{20Mar24} and  Corollary \ref{b20Mar24}. The sets $\pCC_{\mA_1}$ and $\CC_{\mA_1}'$ are described (Theorem \ref{30Mar15}). 
 There are no  descriptions yet of the sets $\pCC_{\mA_n}$ and $\CC_{\mA_n}'$ for $n\geq 2$ but  Theorem \ref{25Mar24} provides explicit left denominators sets   $S\in \Den_l(\mA_n, \ga_n)$ such that  $S^{-1}\mA_n\simeq \pQ (\mA_n)$.

%*** *** We say that $\udim ({}_RR)<\infty$ if there are uniform left ideals $U_1, \ldots , U_n$ of $R$ such that $\oplus_{i=1}^n U_i$ is an essential left ideal of $R$. Then $n=\udim (\RR)$ does not depend on the choice of the uniform left ideals $U_i$ and is called the {\em left uniform dimension} of $R$. Similarly, the {\em right uniform dimension}  $\udim (R_R)$ of $R$ is defined. ***

%%%%%%%%%%%%%%%%%% SECTION 2 %%%%%%%%%%%%%%%%%%%%%

\section{Ore sets, denominator sets and  left or right regular elements  of a ring} \label{ORE-DEN} %\marginpar{ORE-DEN}

 The aim of this section is to present several results on and constructions of left Ore and denominator sets of a ring. Several characterizations of one-sided regular elements of a ring are given in module-theoretic and one-sided-ideal-theoretic way. These results are used in the paper and are of independent interest. 
\\

{\bf Ore and denominator sets, localization of a ring at a denominator set.} Let $R$ be a ring. A subset $S$ of $R$ is called a {\em multiplicative set} if  $SS\subseteq S$, $1\in S$ and $0\not\in S$. A 
multiplicative subset $S$ of $R$   is called  a {\em
left Ore set} if it satisfies the {\em left Ore condition}: for
each $r\in R$ and
 $s\in S$, $$ Sr\bigcap Rs\neq \emptyset .$$
Let $\Ore_l(R)$ be the set of all left Ore sets of $R$.
  For  $S\in \Ore_l(R)$, $\ass_l (S) :=\{ r\in
R\, | \, sr=0 \;\; {\rm for\;  some}\;\; s\in S\}$  is an ideal of
the ring $R$. 

%$\noindent $

A left Ore set $S$ is called a {\em left denominator set} of the
ring $R$ if $rs=0$ for some elements $ r\in R$ and $s\in S$ implies
$tr=0$ for some element $t\in S$, i.e., $r\in \ass_l (S)$. Let
$\Den_l(R)$ (resp., $\Den_l(R, \ga )$) be the set of all left denominator sets of $R$ (resp., such that $\ass_l(S)=\ga$). For
$S\in \Den_l(R)$, let $$S^{-1}R=\{ s^{-1}r\, | \, s\in S, r\in R\}$$
be the {\em left localization} of the ring $R$ at $S$ (the {\em
left quotient ring} of $R$ at $S$). By definition, in Ore's method of localization one can localize {\em precisely} at the left denominator sets.
 In a similar way, right Ore and right denominator sets are defined. 
Let $\Ore_r(R)$ and 
$\Den_r(R)$ be the set of all right  Ore and  right   denominator sets of $R$, respectively.  For $S\in \Ore_r(R)$, the set  $\ass_r(S):=\{ r\in R\, | \, rs=0$ for some $s\in S\}$ is an ideal of $R$. For
$S\in \Den_r(R)$,  $$RS^{-1}=\{ rs^{-1}\, | \, s\in S, r\in R\}$$   is the {\em right localization} of the ring $R$ at $S$. 

Given ring homomorphisms $\nu_A: R\ra A$ and $\nu_B :R\ra B$. A ring homomorphism $f:A\ra B$ is called an $R$-{\em homomorphism} if $\nu_B= f\nu_A$.  A left and right Ore set is called an {\em Ore set}.  Let $\Ore (R)$ and 
$\Den (R)$ be the set of all   Ore and    denominator sets of $R$, respectively. For
$S\in \Den (R)$, $$S^{-1}R\simeq RS^{-1}$$ (an $R$-isomorphism)
 is  the {\em  localization} of the ring $R$ at $S$, and $\ass (S):=\ass_l(S) = \ass_r(S)$. \\ 

{\bf The ring $\RSm$ and the ideal $\ass_R(S)$.} Let $R$ be a ring and $S$ be a  subset of  $R$. Let $R\langle X_S\rangle$ be a ring freely generated by the ring $R$ and a set $X_S=\{ x_s\, | \, s\in S\}$ of free noncommutative indeterminates (indexed by the elements of the set $S$). Let $I_S$ be the  ideal of $R\langle X_S\rangle$  generated by the set  $\{ sx_s-1, x_ss-1 \, | \, s\in S\}$ and 
%\marginpar{RbSbm}
\begin{equation}\label{RbSbm}
\RSm := R\langle X_S\rangle/ I_S.
\end{equation}
The ring $\RSm$ is called the {\bf localization of $R$ at} $S$. 
Let $\ass (S) = \ass_R(S)$ be the  kernel of the ring homomorphism
%\marginpar{RbSbm1}
\begin{equation}\label{RbSbm1}
\s_S: R\ra \RSm , \;\; r\mapsto r+ I_S.
\end{equation}
 The map $\pi_S:R\ra \bR:= R/ \ass_R(S)$, $r\mapsto \br := r+\ass_R(S)$ is   an epimorphism.  The  ideal $\ass_R(S)$ of $R$ has a complex structure, its description is given in \cite[Proposition 2.12]{LocSets} when  $\RSm =\{ \bs^{-1}\br\, | \, s\in S, r\in R\}$ is a ring of left fractions.  We identify the factor ring $\bR $ with its isomorphic copy in the ring $\RSm$  via the monomorphism
%\marginpar{bRS}
\begin{equation}\label{bRS}
\overline{\s_S}: \bR\ra \RSm , \;\; r+\ass_R(S)\mapsto r+ I_S. 
\end{equation}
Clearly, $\bS := (S+\ass_R(S) ) /\ass_R(S)\subseteq \CC_{\RSm}$.  \cite[corollary  2.2]{GenLocSets} shows that the rings $\RSm$ and $\bR\langle \bS^{-1}\rangle$ are $R$-isomorphic. For $S=\emptyset$,  $R\langle \emptyset^{-1}\rangle := R$ and $\ass_R(\emptyset ):=0$. 

\begin{definition} A subset $S$ of a ring $R$ is called a {\bf localizable set} of $R$ if  
$\RSm \neq \{ 0\}$. Let $\mrLR$ be the set of localizable sets of $R$ and 
%\marginpar{SxLzRa}
\begin{equation}\label{SxLzRa}
 \ass \, \mrL (R) :=\{ \ass_R (S) \, | \, S\in \mrL (R) \}.
\end{equation} 
\end{definition}
For an ideal $\ga $ of $R$, let $ \mrLRa :=\{ S\in \mrLR \, | \, \ass_R(S) = \ga \}$. Then 

%\marginpar{SLzRa}
\begin{equation}\label{SLzRa}
\mrLR= \coprod_{\ga \in \ass\,  \mrLR } \mrLRa
\end{equation}
is a disjoint union of non-empty sets. The set $(\mrLR , \subseteq )$ is a partially ordered set (poset)  w.r.t. inclusion $\subseteq $, and $(\mrLRa, \subseteq )$ is a sub-poset of $(\mrLR , \subseteq )$ for every $\ga \in \ass \,\mrLR$.

Proposition \ref{A5Dec22} is the {\em universal property of  localization}. 

\begin{proposition}\label{A5Dec22}%\marginpar{A5Dec22}
Let $R$ be a ring,  $S\in \mrLR$, and $\s_S : R\ra  R\langle S^{-1}\rangle$, $r\mapsto r+\ass_R(S)$. Let $f: R\ra A$ be a ring homomorphism such that $f(S)\subseteq A^\times$. Then there is a unique $R$-homomorphism $f':R\langle S^{-1}\rangle \ra A$ such  $f=f'\s_S$, i.e. the diagram below is commutative 
$$\begin{array}[c]{ccc}
R&\stackrel{\s_S}{\ra} &\RSm\\
&\stackrel{f}{\searrow}&{\downarrow}^{\exists !\, f'}\\
&  & A
\end{array}$$
\end{proposition}

{\bf Every Ore set is a localizable set.} Let $S$ be an Ore set of the ring $R$. Theorem \ref{S10Jan19}  states  that  every Ore set is localizable,  gives an explicit description of the ideal $\ass_R(S)$  and the ring $R\langle S^{-1}\rangle$. 
%So, the localization an any Ore set is an example of the perfect localization.
 Theorem \ref{S10Jan19} also  states that  the ring  $R\langle S^{-1}\rangle$ is $R$-isomorphic to the localization $\bS^{-1}\bR$ of the ring $\bR$ at the denominator set $\bS$ of $\bR$.

\begin{theorem}\label{S10Jan19}%\marginpar{S10Jan19}
Let $R$ be a ring and $S\in \Ore (R)$.
\begin{enumerate}
\item \cite[Theorem 4.15]{Bav-intdifline} Every Ore set is a localizable set. 

\item \cite[Theorem 1.6.(1)]{LocSets} $\ga :=\{ r\in R\, | \, srt=0$ for some elements $s,t\in S\}$ is an ideal of $R$ such that $\ga \neq R$. 

\item \cite[Theorem 1.6.(2)]{LocSets} Let $\pi : R\ra \bR :=R/\ga$, $r\mapsto \br =r+\ga$. Then $\bS :=
\pi (S) \in \Den (\bR , 0)$, $\ga = \ga (S)=\ass_R(S)$, $S\in \mrL (R, \ga )$,  and $S^{-1}R\simeq \bS^{-1}\bR$, an $R$-isomorphism. In particular, every Ore set is localizable. 
%\item \cite[Theorem 1.6.(3)]{LocSets} Let $\gb$ be an ideal of $R$ and $\pi^\dag :R\ra R^\dag :=R/\gb$, $r\mapsto r^\dag =r+\gb$. If $S^\dag := \pi^\dag (S)\in \Den (R^\dag ,0)$ then $\ga \subseteq \gb $ and the map $$\bS^{-1}\bR \ra {S^\dag}^{-1}R^\dag , \;\; \bs^{-1}\br \mapsto {s^\dag}^{-1}r^\dag$$ is a ring epimorphism. 
%\item \cite[Theorem 1.6.(4)]{LocSets} Let $f: R\ra Q$ be a ring homomorphism such that $f(S)\subseteq Q^\times$ and the ring $Q$ is generated by $f(R)$ and $\{ f(s)^{-1}\, | \, s\in S\}$. Then  
%\begin{enumerate}
%\item $\ga \subseteq \ker (f)$ and the map $$\bS^{-1}\bR \ra Q, \;\;  \bs^{-1}\br \mapsto f(s)^{-1}f(r)$$ is a ring epimorphism with kernel $\bS^{-1}(\ker (f) / \ga )$. 
%\item Let $\widetilde{R} = R/\ker (f)$ and $\tpi : R\ra \widetilde{R}$, $r\mapsto r+\ker (f)$. Then $\widetilde{S}:= \tpi (S) \in  \Den (\widetilde{R}, 0)$ and $\widetilde{S}^{-1}\widetilde{R}\simeq Q$, an $\widetilde{R}$-isomorphism. 
%\end{enumerate}
\end{enumerate} 
\end{theorem} 

{\bf  Equivalent conditions of  the left Ore condition and applications.} 
Let $M$ be an $R$-module and $\mE_M$ be the set of essential submodules of $M$.

\begin{lemma}\label{b8Mar24}%\marginpar{b8Mar24}
Let $M$ and $M'$ be $R$-modules. 
\begin{enumerate}

\item For all  $f\in \Hom_R(M,M')$,  $f^{-1}(\mE_{M'}):=\{ f^{-1}(L')\, | \, L'\in \mE_{M'}\}\subseteq \mE_M$.

\item If $M\subseteq M'$ then $\mE_M=\{M\cap L'\,|\, L'\in \mE_{M'}\}$.

\end{enumerate}
\end{lemma}

\begin{proof} 1. Suppose that $f^{-1}(L')\not\in \mE_M$ for some $L'\in \mE_{M'}$. Then we can choose a nonzero submodule, say $N$, of $M$ such that $N\cap f^{-1}(L')\neq \{0\}$. Hence, $f(N)\neq  \{ 0\}$ and $f(N)\cap L'=\{0\}$, a contradiction (since $L'\in \mE_{M'}$). 

2. By statement 1, $\{M\cap L'\,|\, L'\in \mE_{M'}\}\subseteq\mE_M $. Given $L\in \mE_M$, we have to show that $L=M\cap L'$ for some $L'\in \mE_{M'}$.   Let $C$ be the complement of the submodule $L\subseteq M'$. Then the direct sum $L':=L\oplus C$ 
 is an essential submodule of $M'$ such that $M\cap L'=L\oplus (M\cap C)=L$, as required  (since $L$ is an essential submodule of $M$). 
\end{proof}

 For a ring $R$, its left ideal $L$ and an element $r\in R$, the set $(L:r) :=\{ r'\in R\, | \, r'r\in L\}$ is a left ideal of $R$.

Lemma \ref{a10Mar24} gives equivalent conditions to the left Ore condition. They are used in constructions of (new large) classes of left Ore sets  by strengthening some of them (Proposition \ref{A10Mar24}). 

\begin{lemma}\label{a11Mar24}%\marginpar{a11Mar24}
Suppose that $S$ be a multiplicative subset of a ring $R$ and $\ga :=\ass_l(S):=\{ r\in R, | \, sr=0$ for some element $s\in S\}$. Then the following statements  are equivalent:    
\begin{enumerate}

\item  $S\in \Ore_l(R)$ 

\item The set $\ga$ is an ideal of $R$, $\bS:=S+\ga \in \Ore_l(\bR)$ where $\bR :=R/\ga$.

\item The set $\ga$ is an ideal of $R$, $\bR\bs $ is an essential left ideal of $\bR$ for all $\bs \in \bS$, and $\bS\cap (\bR\bs:\br)\neq \emptyset$ for all $\bs \in \bS$ and $\br \in \bR$. Furthermore, for all $\br\in \bR$ and $\bs\in \bS$, the left ideal $(\bR\bs :r)$ is an essential left ideal of $\bR$ (Lemma \ref{b8Mar24}.(1)).

\item For all left ideals $L$ of $R$ such that $L\not\subseteq \ga$ and all $s\in S$, $L\cap Rs\neq \{ 0\}$, and $S\cap (Rs:r)\neq \emptyset$ for all $s\in S$ and $r\in R\backslash \ga$. Furthermore, for all $r\in R\backslash \ga$ and $s\in S$, the left ideal $(Rs:r)$ is an essential left ideal of $R$  (Lemma \ref{b8Mar24}.(1)).

\end{enumerate}

\end{lemma}

\begin{proof} $(1\Rightarrow 2)$ The implication is well-known (and easy to prove).

$(2\Rightarrow 3)$ Suppose that $0\neq \br\in \bR$. Then $0\not\in \bS\br$ (since  $\bS\subseteq \CC_{\bR}$). Hence the left Ore condition for $\bS\in \Ore_l(\bR)$ implies that  $\bS\br\cap \bR\bs\neq \{ 0\}$ for all $\bs\in \bS$.  Hence $\bR\br\cap \bR\bs\neq \{ 0\}$ for all $\bs\in \bS$.  Therefore, $\bR\bs $ is an essential left ideal of $\bR$ for all $\bs \in \bS$ and $\bS\cap (\bR\bs:\br)\neq \emptyset$ for all $0\neq \br \in \bR$ and $\bs\in \bS$. 

If $\br =0$ then $\bS\cap (\bR\bs: 0)=\bS\cap\bR=\bS\neq \emptyset$ for all $\bs\in \bS$.

$(1\Rightarrow 4)$ Suppose that $r\in R\backslash \ga$. Then $0\not\in Sr$.
 Hence the left Ore condition for $S\in \Ore_l(R)$ implies that  $Sr\cap Rs\neq \{ 0\}$ for all $s\in \bS$. Hence,  $Rr\cap Rs\neq \{ 0\}$ for all $s\in \bS$. Therefore, for all left ideals $L$ of $R$ such that $L\not\subseteq \ga$ and all $s\in S$, $L\cap Rs\neq \{ 0\}$, and $S\cap (Rs:r)\neq \emptyset$ for all $r\in R\backslash \ga$.

$(4\Rightarrow 1)$ If $r\in \ga$ then $s'r=0$ for some element $s\in S$, and so $s'r=0=0s$ for all elements $s\in S$.

If $r\not\in \ga$ then $S\cap (Rs:r)\neq \emptyset$ for all $s\in S$, and so $s'r=r's$ for for some elements  $s'\in S$ and $r'\in R$ (that depend on the pair $(s,r))$. It follows that $S\in \Ore_l(R)$.
\end{proof}

\begin{corollary}\label{a10Mar24}%\marginpar{a10Mar24}
Suppose that $S$ is a multiplicative subset of a ring $R$ such that $S\subseteq \CC_R$.  Then   $S\in \Ore_l(R)$ iff    $Rs$  is an essential left ideal of $R$ for all $s\in S$, and $S\cap (Rs:r)\neq \emptyset$ for all $s\in S$ and  $r\in R$. Furthermore, for all $r\in R$ and $s\in S$, the left ideal $(Rs:r)$ is an essential left ideal of $R$.

\end{corollary}

\begin{proof} The corollary follows from the equivalence $(1\Leftrightarrow 3)$ of Lemma \ref{a11Mar24}.
\end{proof}

We strengthen the second condition of Corollary \ref{a10Mar24} to obtain Proposition \ref{A10Mar24}.

\begin{proposition}\label{A10Mar24}%\marginpar{A10Mar24}
Suppose that $S$ is a multiplicative subset of a ring $R$ such that 
\begin{enumerate}

\item[(A)] For every $s\in S$, the left ideal $Rs$ of $R$ is essential, and

\item[(B)] For all essential left ideals $L$ of $R$, $S\cap L\neq \emptyset$. 

\end{enumerate}
Then $S\in \Ore_l(R)$. 

\end{proposition}

\begin{proof} We have to show that the left Ore condition holds for the set $S$: $Sr\cap Rs\neq \emptyset$ for all  $s\in S$ and $r\in R$.  
 By the statement (A), $Rs$ is an essential left ideal of $R$ (Lemma \ref{b8Mar24}.(1)). Then $(Rs:r)$  is also an essential left ideal of $R$. By the statement (B), $S\cap (Rs:r)\neq \emptyset$, i.e. $s'r=r's$ for some elements $s'\in S$ and $r'\in R$, as required. 
\end{proof}

Recall that a ring $R$ is called a {\bf left Goldie ring} if $R$ has finite left uniform dimension and satisfies the a.c.c. on left annihilators. 
The following example shows that Proposition \ref{A10Mar24} covers a lot of ground.

\begin{example} It is known that the conditions (A) and (B) of Proposition \ref{A10Mar24}  hold for  all  semiprime left Goldie rings $R$ and $S=\CC_R$ (the statement (A) follows at once from the fact that the  ring $R$ has finite left uniform dimension and  the statement (B)  is \cite[Proposition 2.3.5.(ii)]{MR}). By Proposition \ref{A10Mar24}, $\CC_R\in \Den_l(R,0)$. This fact is  the crucial step in the proof of Goldie's Theorem. 
\end{example}

\begin{corollary}\label{b11Mar24}%\marginpar{b11Mar24}
Suppose that $S$ is a multiplicative subset of a ring $R$ such that the conditions (A) and (B) of Proposition \ref{A10Mar24} and their right analogues hold. 
Then:
\begin{enumerate}

\item  $S\in \Ore(R)$.

\item $\ga :=\{ r\in R\, | \, srt=0$ for some elements $s,t\in S\}$ is an ideal of $R$ such that $\ga \neq R$. Let $\pi : R\ra \bR :=R/\ga$, $r\mapsto \br =r+\ga$. Then $\bS :=
\pi (S) \in \Den (\bR , 0)$, $\ga = \ga (S)=\ass_R(S)$, $S\in \mrL (R, \ga )$,  and $\RSm \simeq \bS^{-1}\bR$, an $R$-isomorphism.
\end{enumerate}
\end{corollary}

\begin{proof} 1. Statement 1 follows from Proposition \ref{A10Mar24}.

2. Statement 2 follows from Theorem \ref{S10Jan19}.
\end{proof}

{\bf The denominator  sets ${}'\CC_R^{le}$, ${\CC_R'}^{re}$ and  $\CC_R^e$.} For a ring $R$, let 
$R^{le}$ and  $R^{re}$ be the sets that contain elements $r\in R$ such that the left ideal $Rr$ and the  ideal $rR$ are essential, respectively. Let $R^e:=R^{le}\cap R^{re}$. Each element $r\in R$
 determines two maps $r\cdot : R\ra R$, $r'\mapsto rr'$ and $\cdot r: R\ra R$, $r'\mapsto r'r$. An element $r\in R$ is called a {\em left} (resp., {\em right}) {\em regular} if $\ker(\cdot r)=0$ (resp., $\ker(r\cdot )=0$). The sets of all left and right regular elements of the ring $R$ are denoted by 
 ${}'\CC_R$ and $\CC_R'$, respectively.
 Let ${}'\CC_R^{le}:= {}'\CC_R\cap R^{le}$, ${\CC_R'}^{re}:= \CC_R'\cap R^{re}$ and  $\CC_R^e:={}'\CC_R^{le}\cap {\CC_R'}^{re}=\CC_R\cap R^e$.

\begin{proposition}\label{A11Mar24}%\marginpar{A11Mar24}

\begin{enumerate}
\item The sets ${}'\CC_R^{le}$, ${\CC_R'}^{re}$ and  $\CC_R^e$ are multiplicative sets of $R$.

\item Suppose that the set ${}'\CC_R^{le}$ meets all the essential left ideals of the ring $R$. Then ${}'\CC_R^{le}\in \Den_l(R)$.

\item Suppose that the set ${\CC_R'}^{re}$ meets all the essential right ideals of the ring $R$. Then ${\CC_R'}^{re}\in \Den_r(R)$.

\item Suppose that the set $\CC_R^e$ meets all the essential left ideals and essential  right ideals of the ring $R$. Then $\CC_R^e\in \Den (R, 0)$.
\end{enumerate}
\end{proposition}

\begin{proof} 1. It suffices to show that the set ${}'\CC_R^{le}$ is a multiplicative set since then by symmetry the set   ${\CC_R'}^{re}$ is also  multiplicative. These two results imply that   $\CC_R^e$ is  a multiplicative set.

Clearly, $1\in {}'\CC_R^{le}$. Let  $s,t\in {}'\CC_R^{le}$. Then $st\in {}'\CC_R$. It remain to show that $st \in R^{le}$. The left ideal $Rs$ of $R$ is an essential ideal. Then  the $R$-module  $Rst$ is an essential submodule of $Rt$ (since $t\in {}'\CC$). The inclusions of essential left ideals of $R$, $Rst\subseteq Rt\subseteq R$,  imply that the left ideal $Rst$ of $R$ is essential, as required.  

2. By statement 1, the set ${}'\CC_R^{le}$ is a multiplicative set of $R$. Since ${}'\CC_R^{le}\subseteq {}'\CC_R $, it remains to show that  ${}'\CC_R^{le}\in \Ore_l(R)$. This follows from Proposition \ref{A10Mar24} (since the conditions (A) and (B) of Proposition \ref{A10Mar24} are satisfied).

3. Statement 3 follows from statement 2 when apply to the opposite ring of the ring $R$.

4. Statement 4 follows from statements 2 and 3.
\end{proof}

For two multiplicative sets $S$ and $T$ of $R$, let $ST$ be a multiplicative submonoid of $(R,\cdot)$ generated by $S$ and $T$. Clearly, the product is commutative,  $ST=TS$, associative and $\{ 1\} S=S$.  The product $ST$ is a multiplicative set iff $0\not\in ST$.

\begin{lemma}\label{a14Sep13}%\marginpar{a14Sep13}
(\cite[Lemma 2.4]{stronglquot})
\begin{enumerate}
\item Let $S,T\in \Ore_l(R)$. If $0\not\in ST$ then $ST \in \Ore_l (R)$.  \item Let $S,T\in \Den_l(R)$. If $0\not\in ST$ then $ST \in \Den_l (R)$.
    \item Statements 1 and 2 hold also for Ore sets  and denominator sets, respectively.
\end{enumerate}
\end{lemma}

\begin{lemma}\label{a12Mar24}%\marginpar{a12Mar24}
Let $S$ be a multiplicative set of a ring $R$. Then:
\begin{enumerate}

\item The set $S$ contains the largest left/right/left and right Ore set which is denoted by $S^{lO}/S^{rO}/S^O$.

\item  The set $S$ contains the largest left/right/left and right denominator set which is denoted by $S^{ld}/S^{rd}/S^d$.

\item In the set $\pCCR$ every left Ore set is a left denominator set, and vice versa.

\item In the set $\CC_R'$ every right Ore set is a right denominator set, and vice versa.

\item In the set $\CC_R$ every  Ore set is a  denominator set, and vice versa.

\end{enumerate}
\end{lemma}

\begin{proof} 1. Statement 1 follows from Lemma \ref{a14Sep13}.(1) and the largest left/right/left and right Ore set of $S$ is a union of all  left/right/left and right Ore sets in $S$.

2. Statement 2 follows from Lemma \ref{a14Sep13}.(2) and the largest left/right/left and right denominator set of $S$ is a union of all  left/right/left and right denominator sets in $S$.

3-5. Statements 3-5 are obvious.
\end{proof}

 In view of Proposition \ref{A11Mar24}.(1) and   Lemma \ref{a12Mar24}, we have the following definitions.

\begin{definition}
Let ${}'\CC_R^{leO}$ be the largest left Ore/denominator set in the multiplicative set ${}'\CC_R^{le}$. Let ${}'\CC_R^{lee}$ be the largest multiplicative  subset of the multiplicative set ${}'\CC_R^{le}$  that meets all the essential left ideals of $R$. Let ${\CC_R'}^{reO}$ be the largest right Ore/denominator  set in the multiplicative set ${\CC_R'}^{re}$. Let ${\CC_R'}^{ree}$ be the largest multiplicative subset of the multiplicative set ${\CC_R'}^{re}$  that meets all the essential right ideals of $R$.
\end{definition}

Lemma \ref{c12Mar24} describes the sets  ${}'\CC_R^{leO}$, ${\CC_R'}^{reO}$, ${}'\CC_R^{lee}$ and ${\CC_R'}^{ree}$.

\begin{lemma}\label{c12Mar24}%\marginpar{c12Mar24}

\begin{enumerate}

\item The set ${}'\CC_R^{leO}$ (resp., ${\CC_R'}^{reO}$) is the union of all left (resp., right) Ore/denominator  sets in ${}'\CC_R^{leO}$ (resp., ${\CC_R'}^{re}$). In particular, ${}'\CC_R^{leO}\neq \emptyset$ and ${\CC_R'}^{reO}\neq \emptyset$.

\item  ${}'\CC_R^{lee}\in \{ \emptyset, {}'\CC_R^{le} \}$   and ${\CC_R'}^{ree}\in \{ \emptyset, {\CC_R'}^{re}\}$. If ${}'\CC_R^{lee}={}'\CC_R^{le} $   (resp., ${\CC_R'}^{ree}={\CC_R'}^{re}$) then ${}'\CC_R^{lee}={}'\CC_R^{le} \in \Den_l(R)$   (resp., ${\CC_R'}^{ree}={\CC_R'}^{re}\in \Den_r(R)$).
\end{enumerate}
\end{lemma}

\begin{proof} 1. By Proposition \ref{A11Mar24}.(1),  the sets ${}'\CC_R^{le}$ and  ${\CC_R'}^{re}$  are multiplicative sets of $R$. Now, statement 1 follows from Lemma \ref{a12Mar24}.(1,2).

2. Suppose that  ${}'\CC_R^{lee}\neq \emptyset$. Notice that  ${}'\CC_R^{lee}\subseteq {}'\CC_R^{le}$.  Then clearly $ {}'\CC_R^{lee}={}'\CC_R^{le}$. 
Similarly, suppose that ${\CC_R'}^{ree}\neq \emptyset$. Notice that  ${\CC_R'}^{ree}\subseteq {\CC_R'}^{re}$. Then $ {\CC_R'}^{ree}={\CC_R'}^{re}$.  The equalities in statements 2 follows from Proposition \ref{A11Mar24}.(2,3).
\end{proof}

By  Lemma \ref{c12Mar24}, the sets ${}'\CC_R^{leO}$ and ${}'\CC_R^{lee}$ are left denominator sets of the ring $R$ such that ${}'\CC_R^{lee}\subseteq {}'\CC_R^{leO}$ provided ${}'\CC_R^{lee}\neq \emptyset$.
By  Lemma \ref{c12Mar24}, the sets ${\CC_R'}^{reO}$ and ${\CC_R'}^{ree}$ are right denominator sets of the ring $R$ such that ${\CC_R'}^{ree}\subseteq {\CC_R'}^{reO}$  provided ${}'\CC_R^{ree}\neq \emptyset$.

\begin{definition} 
Let ${}'Q(R)^{leO}:=   \Big({}'\CC_R^{leO}\Big)^{-1}R$, ${}'Q(R)^{lee}:=   \Big({}'\CC_R^{lee}\Big)^{-1}R$, ${Q'}^{reO}:=R\Big({\CC_R'}^{reO} \Big)^{-1}$ and ${Q'}^{ree}:=R\Big({\CC_R'}^{ree} \Big)^{-1}$.
\end{definition}

 Recall that  $\pSlR$ is the largest left denominator set in $\pCCR$ and $\pQ_l(R):= \pSlR^{-1}R$ is the left regular left quotient ring of $R$,  $\pga := \ass_R(\pSlR )$ and ${}'\pi : R\ra \bR':= R/ \pga$, $r\mapsto \br := r+\pga$.

If ${}'\CC_R^{lee}\neq \emptyset$ and ${\CC_R'}^{ree}\neq \emptyset$ then, by Proposition \ref{A11Mar24}.(2,3), 
%\marginpar{RpCle}
\begin{equation}\label{RpCle}
R^\times \subseteq {}'\CC_R^{lee}={}'\CC_R^{le}\subseteq {}'\CC_R^{leO}\subseteq {}'S_l(R)\subseteq \pCCR\;\; {\rm and}\;\;
R^\times \subseteq {\CC_R'}^{ree}= {\CC_R'}^{re}\subseteq {\CC_R'}^{reO}\subseteq S_r'(R)\subseteq \CC_R'.
\end{equation}
 Hence there are $R$-homomorphisms:
%\marginpar{RpCle-1}
\begin{equation}\label{RpCle-1}
{}'Q(R)^{leO}\ra {}'Q(R)^{lee}, \;\; s^{1}r\mapsto s^{-1}r\;\; {\rm and}\;\; {Q'}^{reO}\ra {Q'}^{ree}, \;\; rt^{-1}\mapsto rt^{-1}
\end{equation}
with kernels $ \ass_l({}'\CC_R^{leO})/ \ass_l({}'\CC_R^{lee}) $ and $\ass_r({\CC_R'}^{reO})/ \ass_r({\CC_R'}^{ree})$, respectively.\\

{\bf The left/right/two-sided regular sets of a ring $R$ and monomorphism of $R$-modules.} For a    right $R$-module $M_R$, let $\pCC_M:=\{ r\in R \, | \, \ker (\cdot r_M)=0\}$ where $\cdot r_M: M\ra M, \;\; m\ra mr$. Similarly, for a left  $R$-module ${}_RM$, let $\CC_M':=\{ r\in R\, | \, \ker (r_M\cdot )=0\}$ where $ r_M\cdot: M\ra M, \;\; m\ra rm$. For an    $R$-bimodule $M$, let $\CC_M:=
\pCC_M\cap \CC_M'$.

Lemma \ref{a15Mar24} makes connections between the sets $\pCCR$, $\CC_R'$ and $\CC_R$ and $\pCC_M$, $\CC_M'$ and $\CC_M$, respectively (under  faithfulness condition).  

\begin{lemma}\label{a15Mar24}%\marginpar{a15Mar24}
 
\begin{enumerate}

\item If $M_R$ is a faithful right $R$-module then $\pCC_M\subseteq \pCCR$.

\item If ${}_RM$ is a faithful left $R$-module then $\CC_M'\subseteq \CC_R'$.

\item If ${}_RM_R$ is an $R$-module which is faithful as a left and right $R$-module then $\CC_M\subseteq \CC_R$.

\item If $I$ is an ideal of the ring $R$ which is faithful as a left and right $R$-module then $\CC_I\subseteq \CC_R$.

\end{enumerate}
\end{lemma}

\begin{proof} 1. Suppose that $c\in \pCC_M$ but
 $c\not\in \pCCR$. Then $dc=0$ or some nonzero element $d\in R$. By the assumption,  $M_R$ is a faithful right $R$-module. Hence, $m':=md\neq 0$ for some element $m\in M$. Then $0\neq m'c=mdc=0$, a contradiction, and statement 1 follows.

2. By symmetry,  statement 2 follows from statement 1.

3. Statement 3 follows from statements 1 and 2.

4. Statement 4 is a particular case of statement 3. 
\end{proof}

Applying  Lemma \ref{a15Mar24} for faithful  (one-sided) ideals $I$ of a ring $R$ we obtain even tighter  connections between the sets $\pCCR$, $\CC_R'$ and $\CC_R$ and $\pCC_I$, $\CC_I'$ and $\CC_I$, respectively.

\begin{lemma}\label{b15Mar24}%\marginpar{b15Mar24}
 
\begin{enumerate}

\item If $I_R$ is a faithful right ideal of $R$ then $\pCCR= \pCC_I$.

\item If ${}_RI$ is a faithful left ideal of $R$ then $\CC_R'= \CC_I'$.

\item If $I$ is an ideal of  $R$ which is  faithful as a left and right $R$-module then $\CC_R= \CC_I$.

\end{enumerate}
\end{lemma}

\begin{proof} 1. By Lemma \ref{a15Mar24}.(1), $\pCCR\supseteq \pCC_I$. The opposite inclusion follows from the inclusion $I_R\subseteq R$.

2. By symmetry,  statement 2 follows from statement 1.

3. Statement 3 follows from statements 1 and 2. 
\end{proof}

For a module $M$ and its submodule $N$, Lemma \ref{c15Mar24} makes connections between the sets $\pCC_M$, $\CC_M'$ and $\CC_M$ and $\pCC_N$, $\CC_N'$ and $\CC_N$, respectively (under  essentiality  condition).

\begin{lemma}\label{c15Mar24}%\marginpar{c15Mar24}

\begin{enumerate}

\item If $M_R$ is a  right $R$-module and $N_R$ is  an essential submodule of $M$ then $\pCC_M= \pCC_N$.

\item If ${}_RM$ is a  left $R$-module and ${}_RN$ is an essential submodule of $M$ then $\CC_M'= \CC_N'$.

\item If ${}_RM_R$ is an $R$-bimodule and ${}_RN_R$ is an $R$-sub-bimodule such that ${}_RN$ and $N_R$ are essential submodules of ${}_RM$ and $M_R$, respectively. Then $\CC_M=\CC_N$.

\end{enumerate}
\end{lemma}

\begin{proof} 1.

Suppose that $c\in \pCC_M$ but
 $c\not\in \pCCR$. Then $dc=0$ or some nonzero element $d\in R$. By the assumption,  $M_R$ is a faithful right $R$-module. Hence, $m':=md\neq 0$ for some element $m\in M$. Then $0\neq m'c=mdc=0$, a contradiction, and statement 1 follows.

2. By symmetry,  statement 2 follows from statement 1.

3. Statement 3 follows from statements 1 and 2.
\end{proof}

Corollary \ref{d15Mar24} is a particular case of Lemma \ref{c15Mar24}.

\begin{corollary}\label{d15Mar24}%\marginpar{d15Mar24}
 
\begin{enumerate}

\item If $I_R$ is an essential  right ideal of $R$ then $\pCCR= \pCC_I$.

\item If ${}_RI$ is an essential left ideal of $R$ then $\CC_R'= \CC_I'$.

\item If $I$ is an ideal of  $R$ which is  essential as a left and right $R$-module then $\CC_R= \CC_I$.

\end{enumerate}
\end{corollary}

\begin{proof} 1. By Lemma \ref{a15Mar24}.(1), $\pCCR\supseteq \pCC_I$. The opposite inclusion follows from the inclusion $I_R\subseteq R$.

2. By symmetry,  statement 2 follows from statement 1.

3. Statement 3 follows from statements 1 and 2. 
\end{proof}

 Let ${\rm ICS}({}_RM)$ (resp., ${\rm ICS}(M_R)$) be the set of isomorphism classes of simple submodules of a  semisimple left (resp. right) $R$-module $M$. 
Every semisimple left (resp. right) $R$-module $M$ is a unigue direct sum of its isotypic components, $M=\oplus_{[V]\in {\rm ICS}({}_RM)}M_{[V]}$ (resp., $M=\oplus_{[U]\in {\rm ICS}(M_R)}M_{[U]}$),  where $M_{[V]}$ (resp., $M_{U]}$) is the sum of all simple submodules of $M$ that are isomorphic to the module $V$ (resp.,  $U$).

\begin{corollary}\label{e15Mar24}%\marginpar{e15Mar24}
 
\begin{enumerate}

\item If $I_R$ is an essential  right ideal of $R$ which is a semisimple right $R$-module  then $\pCCR= \bigcap_{[U]\in {\rm ICS} (I_R)}\pCC_U$.

\item If ${}_RI$ is an essential left ideal of $R$  which is a semisimple left $R$-module then $\CC_R'= \bigcap_{[V]\in {\rm ICS} ({}_RI)}\CC_V'$.

\item If $I$ is an ideal of  $R$ which is  essential and semisimple as a left and right $R$-module and then $\CC_R=\bigcap_{[U]\in {\rm ICS} (I_R),[V]\in {\rm ICS} ({}_RI)}\pCC_U\cap\CC_V'$.

\end{enumerate}
\end{corollary}

\begin{proof} 1 and 2. Statements 1 and 2 follows from  Statements 1 and 2 of Corollary \ref{d15Mar24}, respectively.

3. Statement 3 follows from statements 1 and 2. 
\end{proof}

%%%%%%%%%%%%%%%%%% SECTION 3 %%%%%%%%%%%%%%%%%%%%%

\section{The rings $\pQ (\mS_n)$, $\pQ(\CI_n)$,  $Q'(\mS_n)$ and $Q'(\CI_n)$} \label{RING-PQmSn} %\marginpar{RING-PQmSn}

In this section, Theorem \ref{19Mar24} and Corollary \ref{a22Mar24}.(1) describe the algebras $\pQ (\mS_n)$ and $Q'(\mS_n)$, respectively. Similarly, Corollary \ref{a21Mar24} describes the algebras $\pQ (\CI_n)$ and $Q'(\CI_n)$. Theorem \ref{A1Apr24} and Theorem \ref{A31Mar24} describe the sets $\pCC_{\mS_1}$ and $\CC_{\mS_1}'$.\\

{\bf The algebra $\mS_n$ of one-sided inverses of a polynomial algebra}. We collect some results on the algebras $\mS_n$ from \cite{shrekalg} that are used in the proofs later. 

\begin{definition} (\cite{shrekalg}) The %{\em Shrek}
{\bf algebra  of one-sided inverses} of $P_n=K[x_1,\ldots , x_n]$, $\mathbb{S}_n$,  is an
algebra generated over a field $K$ %of characteristic zero
by $2n$
elements $x_1, \ldots , x_n, y_n, \ldots , y_n$ subject to  the
defining relations:
$$ y_1x_1=\cdots = y_nx_n=1 , \;\; [x_i, y_j]=[x_i, x_j]= [y_i,y_j]=0
\;\; {\rm for\; all}\; i\neq j,$$ where $[a,b]:= ab-ba$, the
commutator of elements $a$ and $b$.
\end{definition}
By the very definition, the algebra $\mS_n$ is obtained from the
polynomial algebra $P_n$ by adding commuting, left (or right)
inverses of its canonical generators. Clearly, $\mathbb{S}_n=\mS_1(1)\t \cdots \t\mS_1(n)\simeq
\mathbb{S}_1^{\t n}$ where $\mS_1(i):=K\langle x_i,y_i \, | \,
y_ix_i=1\rangle \simeq \mS_1$ and $$\mS_n=\bigoplus_{\alpha ,
\beta \in \N^n} Kx^\alpha y^\beta$$ where $x^\alpha :=
x_1^{\alpha_1} \cdots x_n^{\alpha_n}$, $\alpha = (\alpha_1, \ldots
, \alpha_n)$, $y^\beta := y_1^{\beta_1} \cdots y_n^{\beta_n}$ and
$\beta = (\beta_1,\ldots , \beta_n)$. In particular, the algebra
$\mS_n$ contains two polynomial subalgebras $P_n=K[x_1, \ldots , x_n]$ and $\mY_n:=K[y_1,
\ldots , y_n]$ and is equal,  as a vector space,  to their tensor
product $P_n\t \mY_n$. Note that also the Weyl algebra $A_n$ is a
tensor product (as a vector space) $P_n\t K[\der_1, \ldots ,
\der_n]$ of  two polynomial subalgebras.

When $n=1$, we usually drop the subscript `1' if this does not
lead to confusion.  So, $\mS_1= K\langle x,y\, | \,
yx=1\rangle=\bigoplus_{i,j\geq 0}Kx^iy^j$. For each natural number
$d\geq 1$, let $M_d(K):=\bigoplus_{i,j=0}^{d-1}KE_{ij}$ be the
algebra of $d$-dimensional matrices where $\{ E_{ij}\}$ are the
matrix units, and
$$M_\infty (K) :=
\varinjlim M_d(K)=\bigoplus_{i,j\in \N}KE_{ij}$$ be the algebra
(without 1) of infinite dimensional matrices. The algebra $\mS_1$
contains the ideal $F:=\bigoplus_{i,j\in \N}KE_{ij}$, where
%\marginpar{Eijc}
\begin{equation}\label{Eijc}
E_{ij}:= x^iy^j-x^{i+1}y^{j+1}, \;\; i,j\geq 0.
\end{equation}
For all natural numbers $i$, $j$, $k$, and $l$,
$E_{ij}E_{kl}=\d_{jk}E_{il}$ where $\d_{jk}$ is the Kronecker
delta function.  The ideal $F$ is an algebra (without 1)
isomorphic to the algebra $M_\infty (K)$ via $E_{ij}\mapsto
E_{ij}$.  For all $i,j\geq 0$, %\marginpar{xyEij}
\begin{equation}\label{xyEij}
xE_{ij}=E_{i+1, j}, \;\; yE_{ij} = E_{i-1, j}\;\;\; (E_{-1,j}:=0),
\end{equation}
%\marginpar{xyEij1}
\begin{equation}\label{xyEij1}
E_{ij}x=E_{i, j-1}, \;\; E_{ij}y = E_{i, j+1} \;\;\;
(E_{i,-1}:=0).
\end{equation}
The algebra
 %\marginpar{mS1d}
\begin{equation}\label{mS1d}
\mS_1= K\oplus xK[x]\oplus yK[y]\oplus F
\end{equation}
is the direct sum of vector spaces. Then %\marginpar{mS1d1}
\begin{equation}\label{mS1d1}
\mS_1/F\simeq L_1:=K[x,x^{-1}], \;\; x\mapsto x, \;\; y \mapsto
x^{-1},
\end{equation}
since $yx=1$, $xy=1-E_{00}$ and $E_{00}\in F$.

The algebra $\mS_n = \bigotimes_{i=1}^n \mS_1(i)$ contains the
ideal
$$F_n:= F^{\t n }=\bigoplus_{\alpha , \beta \in
\N^n}KE_{\alpha \beta}, \;\; {\rm where}\;\; E_{\alpha
\beta}:=\prod_{i=1}^n E_{\alpha_i \beta_i}(i), \;\; E_{\alpha_i
\beta_i}(i):=x_i^{\alpha_i}y_i^{\beta_i}-x_i^{\alpha_i+1}y_i^{\beta_i+1}.$$
Note that $E_{\alpha \beta}E_{\g \rho}=\d_{\beta \g }E_{\alpha
\rho}$ for all elements $\alpha, \beta , \g , \rho \in \N^n$ where
$\d_{\beta
 \g }$ is the Kronecker delta function;  $F_n=\bigotimes_{i=1}^nF(i)$ and
 $F(i):=\bigoplus_{s,t\in \N}KE_{st}(i)$.\\

{\bf The involution $\eta$ on $\mS_n$}. The algebra $\mS_n$ admits
the {\em involution}
%\marginpar{etamSn}
\begin{equation}\label{etamSn}
\eta : \mS_n\ra \mS_n, \;\; x_i\mapsto y_i, \;\; y_i\mapsto
x_i, \;\; i=1, \ldots , n.
\end{equation}
 It is a $K$-algebra anti-isomorphism
($\eta (ab) = \eta (b) \eta (a)$ for all $a,b\in \mS_n$) such that
$\eta^2 = \id_{\mS_n}$, the identity map on $\mS_n$. So, the
algebra $\mS_n$ is {\em self-dual} (i.e. it is isomorphic to its
opposite algebra, $\eta : \mS_n\simeq \mS_n^{op}$). The involution
$\eta$ acts on the `matrix' ring
$F_n$ as the transposition,  %\marginpar{eEij1}
\begin{equation}\label{eEij1}
\eta (E_{\alpha \beta} )=E_{\beta \alpha}.
\end{equation}
 The canonical generators $x_i$,
$y_j$ $(1\leq i,j\leq n)$ determine the ascending filtration $\{
\mS_{n, \leq i}\}_{i\in \N}$ on the algebra $\mS_n$ in the obvious
way (i.e. by the total degree of the generators): $\mS_{n, \leq
i}:= \bigoplus_{|\alpha |+|\beta |\leq i} Kx^\alpha y^\beta$ where
$|\alpha | \; = \alpha_1+\cdots + \alpha_n$ ($\mS_{n, \leq
i}\mS_{n, \leq j}\subseteq \mS_{n, \leq i+j}$ for all $i,j\geq
0$). Then $\dim (\mS_{n,\leq i})={i+2n \choose 2n}$ for $i\geq 0$,
and so the Gelfand-Kirillov dimension $\GK (\mS_n )$ of the
algebra $\mS_n$ is equal to $2n$. It is not difficult to show
 that the algebra $\mS_n$ is neither left nor
right Noetherian. Moreover, it contains infinite direct sums of
left and right ideals (see \cite{shrekalg}).\\

{\bf The set of height 1 primes of $\mS_n$}.   Consider the ideals
of the algebra $\mS_n$:
$$\gp_1:=F\t \mS_{n-1},\; \gp_2:= \mS_1\t F\t \mS_{n-2}, \ldots ,
 \gp_n:= \mS_{n-1} \t F.$$ Then $\mS_n/\gp_i\simeq
\mS_{n-1}\t (\mS_1/F) \simeq  \mS_{n-1}\t K[x_i, x_i^{-1}]$ and
$\bigcap_{i=1}^n \gp_i = \prod_{i=1}^n \gp_i =F^{\t n }=F_n$.
Clearly, $\gp_i\not\subseteq \gp_j$ for all $i\neq j$.

\begin{itemize}
 \item
{\em The set $\CH_1$ of height one prime ideals of the algebra
$\mS_n$ is} $\{ \gp_1, \ldots , \gp_n\}$.
\end{itemize}

 Let
$\ga_n:= \gp_1+\cdots +\gp_n$. Then the factor algebra
%\marginpar{SnSn}
\begin{equation}\label{SnSn}
\mS_n/ \ga_n\simeq (\mS_1/F)^{\t n } \simeq L_n:=\bigotimes_{i=1}^n
K[x_i, x_i^{-1}]= K[x_1, x_1^{-1}, \ldots , x_n, x_n^{-1}]
\end{equation}
is a  Laurent polynomial algebra in $n$ variables,  and so $\ga_n$
is a prime ideal of height and co-height $n$ of the algebra
$\mS_n$. $$S_y:=\{ y^\alpha\, | \, \alpha \in \N^n\}\in \Den_l(\mS_n, \ga_n)\;\; {\rm  an}\;\;S_y^{-1}\mS_n\simeq \mS_n/\ga_n=L_n.$$ 
The proof of the following statements can be found in \cite{shrekalg}.

\begin{itemize}
 \item   {\em The
 algebra $\mS_n$ is central, prime and catenary. Every nonzero
 ideal  of $\mS_n$ is an essential left and right submodule of}
 $\mS_n$.
 \item  {\em The ideals of
 $\mS_n$ commute ($IJ=JI$);  and the set of ideals of $\mS_n$ satisfy the a.c.c..}
 \item  {\em The classical Krull dimension $\clKdim (\mS_n)$ of $\mS_n$ is $2n$.}
  \item  {\em Let $I$
be an ideal of $\mS_n$. Then the factor algebra $\mS_n / I$ is
left (or right) Noetherian iff $\ga_n\subseteq I$.}
\end{itemize}

\begin{proposition}\label{a19Dec8}%\marginpar{a19Dec8}
{\rm \cite[Corollary 2.2]{shrekalg}} The polynomial algebra $P_n$
 is the only (up to isomorphism)   simple  faithful left $\mS_n$-module.
\end{proposition}

In more detail, ${}_{\mS_n}P_n\simeq \mS_n / (\sum_{i=0}^n \mS_n
y_i) =\bigoplus_{\alpha \in \N^n} Kx^\alpha \overline{1}$,
$\overline{1}:= 1+\sum_{i=1}^n \mS_ny_i$; and the action of the
canonical generators of the algebra $\mS_n$ on the polynomial
algebra $P_n$ is given by the rule:

%\marginpar{Pnac}
\begin{equation}\label{Ppac}
x_i*x^\alpha = x^{\alpha + e_i}, \;\; y_i*x^\alpha = \begin{cases}
x^{\alpha - e_i}& \text{if } \; \alpha_i>0,\\
0& \text{if }\; \alpha_i=0,\\
\end{cases}  \;\; {\rm and }\;\; E_{\beta \g}*x^\alpha = \d_{\g
\alpha} x^\beta,
\end{equation}
where the set $e_1:= (1,0,\ldots , 0),  \ldots , e_n:=(0, \ldots ,
0,1)$ is the canonical basis for the free $\Z$-module $\Z^n$.  We
identify the algebra $\mS_n$ with its image in the algebra
$\End_K(P_n)$ of all the $K$-linear maps from the vector space
$P_n$ to itself, i.e. $\mS_n \subset \End_K(P_n)$.

\begin{corollary}\label{Pnp-a19Dec8}%\marginpar{Pnp-a19Dec8}
 The polynomial algebra $P_n':=\eta (P_n)=K[y_1, \ldots , y_n]$
 is the only (up to isomorphism)    simple faithful right $\mS_n$-module.
\end{corollary}
\begin{proof} In view of the involution $\eta$, the corollary follows from Proposition \ref{a19Dec8}. 
\end{proof}

In more detail, $(P_n')_{\mS_n}\simeq \mS_n / (\sum_{i=0}^n x_i\mS_n) =\bigoplus_{\alpha \in \N^n} \widetilde{1}Ky^\alpha $,
$\widetilde{1}:= 1+\sum_{i=1}^n x_i\mS_n$; and the action of the
canonical generators of the algebra $\mS_n$ on the polynomial
algebra $P_n'$ is given by the rule:
%\marginpar{Pnac-1}
\begin{equation}\label{Ppac-1}
y^\alpha* y_i = y^{\alpha + e_i}, \;\;y^\alpha * x_i = \begin{cases}
y^{\alpha - e_i}& \text{if } \; \alpha_i>0,\\
0& \text{if }\; \alpha_i=0,\\
\end{cases}  \;\; {\rm and }\;\;y^\alpha * E_{\beta \g}= \d_{
\alpha\beta} y^\g.
\end{equation}
%where the set $e_1:= (1,0,\ldots , 0),  \ldots , e_n:=(0, \ldots ,
%0,1)$ is the canonical basis for the free $\Z$-module $\Z^n$.  
%We identify the algebra $\mS_n$ with its image in the algebra
%$\End_K(P_n)$ of all the $K$-linear maps from the vector space
%$P_n$ to itself, i.e. $\mS_n \subset \End_K(P_n)$.

\begin{proposition}\label{B1Apr24}%\marginpar{B1Apr24}
(\cite[Proposition 3]{shrekalg})
\begin{enumerate}
\item  ${}_{\mS_n}F_n\simeq P_n^{(\N^n)}$.

\item $(F_n)_{\mS_n}\simeq (P_n')^{(\N^n)}$.
\end{enumerate}

 \end{proposition}

{\bf Constructions of Ore and denominator sets.}
 Below we collect and prove some useful results that are used in the proofs. They also are of independent interest.

\begin{lemma}\label{a7Mar15}%\marginpar{a7Mar15}
(\cite[Lemma 2.5]{Clas-lreg-quot}) Suppose that $S,T\in \Den_l(R)$ and $S\subseteq T$. Then  the map $\v : S^{-1}R\ra S^{-1}T$, $s^{-1}r\mapsto s^{-1}r$, is a ring homomorphism (where $s\in S$ and $r\in R$).
\begin{enumerate}
\item $\v$  is a monomorphism iff $\ass_R(S)=\ass_R(T)$.
\item $\v$  is a epimorphism iff for each $t\in T$ there exists an element $r\in R$ such that $rt\in S+\ass_R(T)$.
\item $\v$  is a isomorphism iff $\ass_R(S)=\ass_R(T)$ and for each element $t\in T$ there exists an element $r\in R$ such that $rt\in S$.
\item If, in addition, $T\subseteq \pCCR$, then $\v$  is a isomorphism iff $\ass_R(S)=\ass_R(T)$ and for each element $t\in T$ there exists an element $r\in \pCCR$ such that $rt\in S$.
\end{enumerate}
\end{lemma}

\begin{lemma}\label{c16Mar15}%\marginpar{c16Mar15}
(\cite[Lemma 6.1]{Clas-lreg-quot})  Suppose that $T\in \Den_l(R)$ and $S$ be a multiplicative set of $R$ such that $S\subseteq T$, $\ass_R(S) = \ass_R(T)$ and for each element $t\in T$ there exists an element $r\in R$ such that $rt\in S+\ass_R(T)$. Then $S\in \Den_l(R)$ and $S^{-1}R\simeq T^{-1}R$.
\end{lemma}

\begin{definition}
The pair $(S,T)$ that satisfies the conditions of Lemma \ref{c16Mar15}  is called a {\bf left localization pair} of the ring $R$. 

\end{definition}

Lemma \ref{a18Mar24} provides sufficient conditions for the pre-image of a left Ore/denominator set being a left Ore/denominator set.

\begin{lemma}\label{a18Mar24}%\marginpar{a18Mar24}
Let $R$  be a ring, $\ga$ be an  ideal of $R$,
$\pi_\ga :R\ra \bR :=R/\ga$, $r\mapsto \br=r+\ga$, $\overline{\gb}$ be an ideal of $\bR$,  $\gb =\pi_\ga^{-1}(\overline{\gb})$, $\bS$ be  a multiplicative subset of $\bR$ and $S=\pi_\ga^{-1}(\bS)$. 

\begin{enumerate}

\item If $\bS\in \Ore_l(\bR, \overline{\gb})$ and $\ass_l(S)=\gb$ then $S\in \Ore_l(R, \gb)$.

\item If $\bS\in \Den_l(\bR, \overline{\gb})$ and $\ass_l(S)=\gb$   then $S\in \Den_l(R, \gb)$.

\item If $\bS\in \Den (\bR, \overline{\gb})$ and $\ass_l(S)=\ass_r(S)=\gb$   then $S\in \Den(R, \gb)$.

\item If  $\bS\in \Ore_l(\bR, \overline{\gb})$ and $T\subseteq S$ for some $T\in \Ore_l(R, \gb)$ then $S\in \Ore_l(R, \gb)$.

\item If  $\bS\in \Den_l(\bR, \overline{\gb})$ and $T\subseteq S$ for some $T\in \Ore_l(R, \gb)$ then $S\in \Den_l(R, \gb)$.

\end{enumerate}
\end{lemma}

\begin{proof} Clearly, the set $S$ is a multiplicative subset in $R$.

1. We have to show that the left Ore condition holds for the multiplicative subset $S$ of the ring $R$.
 For elements $s\in S$ and $r\in R$, $\overline{s'} \br = \overline{r'}\bs$ for   some elements $s'\in S$ and $r'\in R$ (since $\bS\in \Ore_l(\bR, \overline{\gb})$). Therefore, $d:=s'r-r's\in \gb$. By the assumption $\ass_l(S)=\gb$. So,  there is an element $s''\in S$, such that $0=s''d=s''s'r-s''r's$, i.e. the  left Ore condition holds for $S$.

2. By statement 1,  $S\in \Ore_l(R, \gb)$. It remains to show that $\ass_r(S)\subseteq \ass _l(S)=\gb$. Suppose that  $r\in \ass_r(S)$, i.e. $rs=0$ for some element $s\in S$. Then $\br\bs=0$. It follows that  $\br\in \overline{\gb}$ (since $\bS\in \Den_l(\bR, \overline{\gb})$), and so $r\in \pi_\ga^{-1} (\overline{\gb})=\gb =\ass_l(S)$. Hence, $\ass_r(S)\subseteq \ass _l(S)=\gb$.

3. Statement 3 follows from statement 2 and its right analogue.

4. Since  $\bS\in \Ore_l(\bR, \overline{\gb})$, $\ass_l(S)\subseteq \gb$. Since $T\subseteq S$ and  $T\in \Ore_l(R, \gb)$,  $\gb = \ass_l(T)\subseteq \ass_l(S)$. Hence,  $\ass_l(S)=\gb$.    Now, $S\in \Ore_l(R, \gb)$, by statement 1.

5. By statement 4, $S\in \Ore_l(R,\gb)$. It remains to show that $\ass_r(S)\subseteq \ass _l(S)=\gb$. Suppose that  $r\in \ass_r(S)$, i.e. $rs=0$ for some element $s\in S$. Then $\br\bs=0$ and so $\br\in \overline{\gb}$ (since $\bS\in \Den_l(\bR, \overline{\gb})$), and so $r\in \pi_\ga^{-1} (\overline{\gb})=\gb =\ass_l(S)$. Therefore,  $\ass_r(S)\subseteq \ass _l(S)=\gb$. 
\end{proof}

Now, we obtain a useful corollary.

\begin{corollary}\label{a19Mar24}%\marginpar{a19Mar24}
Let $R$  be a ring, $\ga$ be an  ideal of $R$ and 
$\pi_\ga :R\ra \bR :=R/\ga$, $r\mapsto \br=r+\ga$. 

\begin{enumerate}

\item If $S\in \Ore_l(R, \ga)$ then $S+\ga \in \Ore_l(R, \ga)$.

\item If $S\in \Den_l(R, \ga)$ then $S+\ga \in \Den_l(R, \ga)$ and $(S+\ga)^{-1}R\simeq S^{-1}R$.

\item If $S\in \Den(R, \ga)$ then $S+\ga \in \Den(R, \ga)$ and $(S+\ga)^{-1}R\simeq S^{-1}R$.

\end{enumerate}
\end{corollary}

\begin{proof} 1 and 2. We keep the notation of Lemma \ref{a18Mar24}. Suppose that $S\in \Ore_l(R, \ga)/\Den_l(R, \ga)$. Then  $S':=S+\ga = \pi_{\ga}^{-1}(S)$ is a multiplicative subset of $R$ that contains $S$. Therefore, $\ga = \ass_l(S)\subseteq \ass_l(S')$. Let $r\in \ass_l(S')$. Then $(s+a)r=0$ for some elements $s\in S$ and $a\in \ga$, and so $sr=-ar\in \ga$. There is an element $t\in S$  such that $tsr=0$, i.e. $r\in \ga$. Therefore,  $\ass_l(S)=\ga = \ass_l(S')$. Clearly, $\bS:=\pi_\ga (S)\in \Ore_l(\bR , 0)/\Den_l(\bR, 0)$. Now, by Lemma \ref{a18Mar24}.(1,2), $S'\in \Ore_l(R, \ga)/\Den_l(R, \ga)$. If $S'\in \Den_l(R, \ga)$ then 
${S'}^{-1}R\simeq \pi_\ga (S')^{-1}\bR=\pi_\ga (S)^{-1}\bR\simeq S^{-1}R$.

3. Statement 3 follows from statement 2 and its right analogue.
\end{proof}

Lemma \ref{b18Mar24} provides sufficient conditions for  a left Ore/denominator of a larger ring being a left Ore/denominator set of a smaller ring which contains it. 

\begin{lemma}\label{b18Mar24}%\marginpar{b18Mar24}
Let  $R$  be a subring of a ring $R'$ and  $S$ be a multiplicative subset of $R$ such that the left $R$-module $R'/R$ is $S$-torsion (for each $r'\in R'$ there is an element  $s\in S$ such that $sr'\in R$).

\begin{enumerate}

\item If $S\in \Ore_l(R', \ga')$  then $S\in \Ore_l(R, R\cap \ga')$. 

\item If $S\in \Den_l(R', \ga')$  then $S\in \Den_l(R, R\cap \ga')$ and $S^{-1}R\simeq S^{-1}R'$.

\end{enumerate}
\end{lemma}

\begin{proof} 1. (i) $S\in \Ore_l(R)$: For each element $s\in S$ and $r\in R$, $s'r=r's$ for some elements $s'\in S$ and $r'\in R'$ (since $S\in \Ore_l(R', \ga')$). By the assumption, the left $R$-module $R'/R$ is $S$-torsion, and so $tr'\in R$ for some element $t\in S$. Now,  $ts'r=tr's$ where $ts'\in S$ and $tr'\in R$, and the $S\in \Ore_l(R)$. 

(ii) $\ass_{l, R}(S)=R\cap \ga'$: Since $R\subseteq R'$ and  $S\in \Ore_l(R', \ga')$, $\ass_{l,R}(S)=R\cap \ass_{l,R'}(S)=R\cap \ga'$.

2. (i) $S\in \Den_l(R, R\cap \ga')$: 
By statement 1, $S\in \Ore_l(R, R\cap \ga')$. It remain to show that $\ass_r(S)\subseteq R\cap \ga'$. Given an element $r\in R$ such that $rs=0$ for some $s\in S$. Then $r\in \ga'$ (since $S\in \Den_l(R',\ga')$), and so $r\in R\cap \ga'$, as required. 

(ii) $S^{-1}R\simeq S^{-1}R'$: By the statement (i),  the map $\phi: S^{-1}R\ra S^{-1}R'$, $s^{-1}r\mapsto s^{-1}r$ is a ring homomorphism. Suppose that $a:=s^{-1}r\in \ker (\phi)$. Then $r\in \ga'$, and so $r\in R\cap \ga'$, i.e. $a=0$. So, the map $\phi$ is a monomorphism. It remains to show that the map $\phi$ is an epimorphism. Given element $s^{-1}r'\in S^{-1}R'$ where $s\in S$ and $r'\in R'$. Since the $R$-module $R'/R$ is $S$-torsion, there is an element $t\in S$ such that $r:=tr'\in R$. Then  $s^{-1}r'=(ts)^{-1}tr'=(ts)^{-1}r\in S^{-1}R$. Therefore, the map $\phi $ is epimorphism and the statement (ii) follows. 
\end{proof}

{\bf The ring $\pQ (\mS_n)$.} Recall that $\mY_n=K[y_1, \ldots , y_n]$ is a polynomial algebra. Let 
$\mY_n^0:=\mY_n\backslash \{ 0\}$, ${}'\mY_n^0:=\mY_n\cap \pCC_{\mS_n}$, 
$$
{}'\mY_n:=\mY_n\cap \pS (\mS_n)\;\; {\rm and}\;\;  
\widetilde{{}'\mY_n}:=\{c\in \mY_n\, | \, y^\alpha c\in \pS (\mS_n)\; {\rm  for\; some}\;\;\alpha \in \N^n\},
$$
$\CT=\pi_{\ga_n}^{-1}(L_n\backslash \{ 0\})=\mS_n\backslash \ga_n$ and  $\CS:=\pi_{\ga_n}^{-1}(\mY_n\backslash \{ 0\}) =\mY_n\backslash \{ 0\}+\ga_n$ where $\pi_{\ga_n}: \mS_n\ra \mS_n/\ga_n =L_n$, $r\mapsto \br:= r+\ga_n$.
 Clearly, ${}'\mY_n\subseteq 
{}'\mY_n^0\subseteq \mY_n^0$ and $\CS\subseteq \CT$.

Theorem \ref{19Mar24} describes the algebra $\pQ (\mS_n)$.

\begin{theorem}\label{19Mar24}%\marginpar{19Mar24}
 
 \begin{enumerate}
 
 \item $\pQ (\mS_n)\simeq K(y_1, \ldots , y_n)$.

 \item ${}'\mY_n\in \Den_l(\mS_n, \ga_n)$ and ${}'\mY_n^{-1}\mS_n\simeq \pQ (\mS_n)$. Furthermore, the  subset  ${}'\mY_n$ of $\mY_n$  is a left denominator set of $\mS_n$ which is the  largest  left denominator set that is contained in the multiplicative set $ \mY_n \cap\pCC_{\mS_n}$.

\item $\CS, \CT\in \Den_l(\mS_n, \ga_n)$ and ${}'\mY_n^{-1}\mS_n\simeq \CT^{-1}\mS_n\simeq \CS^{-1}\mS_n\simeq  K(y_1, \ldots , y_n)$.

 \end{enumerate}
  
\end{theorem}

\begin{proof} 1. Statement 1 follows from statements 2 and 3.  

2.  (i) {\em The pair $({}'\mY_n,\pS (\mS_n))$ is a left localization pair of the ring $\mS_n$}:
 By the definition, the set    ${}'\mY_n$ is a multiplicative subset of $\pS (\mS_n)\subseteq \mS_n$. It follows from the inclusions $S_y \subseteq {}'\mY_n\subseteq \pS (\mS_n)$ that 
$$\ga_n = \ass_l(S_y )\subseteq \ass_l({}'\mY_n)\subseteq \ass_l(\pS (\mS_n))=\ga_n, $$
and so $\ass_l({}'\mY_n)=\ga_n=\ass_l(\pS (\mS_n))$. For each element $s\in \pS (\mS_n)$, there is an element $\alpha\in \N^n$ such that $y^\alpha s\in {}'\mY_n$. Notice that $y^\alpha \in S_y \subseteq {}'\mY_n$, and the statement (i) follows.

(ii)  ${}'\mY_n\in \Den_l(\mS_n, \ga_n)$ {\em and} ${}'\mY_n^{-1}\mS_n\simeq\pQ (\mS_n)$: The statement (i) follows from the statement (i) and Lemma \ref{c16Mar15} where   $T=\pS (\mS_n)\in \Den_l(\mS_n,\ga_n)$ and  $S={}'\mY_n$. 

(iii) {\em The  subset  ${}'\mY_n$ of $\mS_n$  is a left denominator set of $\mS_n$ which is the  largest  left denominator set that is contained in the multiplicative set $\mY_n \cap\pCC_{\mS_n}$}: Let $T$ be a left denominator set of $\mS_n$ which is the  largest  left denominator set that is contained in the multiplicative set $\mY_n \cap\pCC_{\mS_n}$. By the statement (ii),  ${}'\mY_n\in \Den_l(\mS_n, \ga_n)$. Clearly,  
$$
{}'\mY_n=\mY_n \cap\pS (\mS_n)\subseteq  \mY_n \cap\pCC_{\mS_n},
$$
and so ${}'\mY_n\subseteq T$. Since $T\subseteq \pCC_{\mS_n}$ and $\pS (\mS_n)$ is a the largest left denominator set in $\pCC_{\mS_n}$, we have the inclusion $T\subseteq \mY_n\cap\pS(\mS_n)={}'\mY_n$. Therefore, $T={}'\mY_n$.

3. (i) $\CT\in \Den_l(\mS_n, \ga_n)$ {\em and} $\CT^{-1}\mS_n\simeq K(y_1, \ldots , y_n)$: Since $S_y\in \Den_l(\mS_n, \ga_n)$ and $S_y\subseteq \CT$, $\ga_n=\ass_l(S_y)\subseteq \ass_l(\CT)$. The opposite inclusion follows from the fact that the  factor ring $\mS_n/\ga_n\simeq K[y_1^{\pm 1},\ldots , y_n^{\pm 1}]$ is a domain. So, $\ass_l(\CT)=\ga_n$. Notice that 
$$
\pi_{\ga_n}(\CT)= 
K[y_1^{\pm 1},\ldots , y_n^{\pm 1}]\backslash \{ 0\}\in \Den_l(K[y_1^{\pm 1},\ldots , y_n^{\pm 1}],0).$$
By Lemma \ref{a19Mar24}.(2), $\CT\in \Den_l(\mS_n, \ga_n)$ and 
$$\CT^{-1}\mS_n\simeq \pi_{\ga_n}(\CT)^{-1}(\mS_n/\ga_n)\simeq \pi_{\ga_n}(\CT)^{-1}(K[y_1^{\pm 1},\ldots , y_n^{\pm 1}])\simeq K(y_1,\ldots , y_n).
$$

(ii) {\em The pair $({}'\mY_n,\CT)$ is a left localization pair of the ring $\mS_n$}: By statement 2, ${}'\mY_n\in \Den_l(\mS_n, \ga_n)$. 
 For each element $t\in \CT$, there is an element $\alpha\in \N^n$ such that $y^\alpha t\in {}'\mY_n$. Notice that $y^\alpha \in S_y \subseteq {}'\mY_n$, and the statement (ii) follows.

(iii)   ${}'\mY_n^{-1}\mS_n\simeq\CT^{-1}\mS_n$: The statement (iii) follows from the statement (ii) and Lemma \ref{c16Mar15} where   $T=\CT\in \Den_l(\mS_n,\ga_n)$ and  $S={}'\mY_n$.

(iv) {\em The pair $(\CS,\CT)$ is a left localization pair of the ring $\mS_n$}: By the  statement (i), $\CT\in \Den_l(\mS_n, \ga_n)$. 
 By the definition, the set    $\CS$ is a multiplicative subset of $ \mS_n$. The inclusions $S_y \subseteq \CS\subseteq \CT$ imply that 
$$\ga_n = \ass_l(S_y )\subseteq \ass_l(\CS)\subseteq \ass_l(\CT)=\ga_n, $$
and so $\ass_l(\CS)=\ga_n=\ass_l(\CT)$. 
 For each element $t\in \CT$, there is an element $\alpha\in \N^n$ such that $y^\alpha t\in \CS$. Notice that $y^\alpha \in S_y \subseteq\CS$, and the statement (iv) follows.

(v)   $\CS^{-1}\mS_n\simeq\CT^{-1}\mS_n$: The statement (v) follows from the statement (iv) and Lemma \ref{c16Mar15} where   $T=\CT\in \Den_l(\mS_n,\ga_n)$ and  $S=\CS$.

Now, statement 3 follows. 
\end{proof}

\begin{theorem}\label{8Mar15}%\marginpar{8Mar15}
(\cite[Theorem 4.3]{Clas-lreg-quot}) Let $R$ be a ring. Then
\begin{enumerate}

\item $\pQlR$ is a left Artinian ring iff $\pQ_{l,cl}(R)$ is a left Artinian ring. If one of the equivalent conditions holds then $\pSlR = \pCCR$ and $\pQlR = \pQ_{l, cl}(R)$.

\item $\pQlR$ is a semisimple  Artinian ring iff $\pQ_{l,cl}(R)$ is a semisimple Artinian ring. If one of the equivalent conditions holds then $\pSlR = \pCCR$ and $\pQlR = \pQ_{l, cl}(R)$.

\end{enumerate}
\end{theorem}

\begin{corollary}\label{a31Mar24}%\marginpar{a31Mar24}
$\pCC_{\mS_n}=\pS (\mS_n)$ and ${}'Q_{l,cl}(\mS_n)={}'Q(\mS_n)\simeq K(y_1, \ldots , y_n)$.
\end{corollary}

\begin{proof} By Theorem \ref{19Mar24}, the ring $ {}'Q(\mS_n)\simeq K(y_1, \ldots , y_n)$ is a field. In particular, it is a semisimple Artinian ring. Now, the corollary follows from Theorem \ref{8Mar15}.(2). 
\end{proof}

 Theorem \ref{19Mar24} and Corollary \ref{a20Mar24} produce explicit left denominators sets   $S\in \Den_l(\mS_n, \ga_n)$ such that  $S^{-1}\mI_n\simeq \pQ (\mS_n)$. 
  By Theorem \ref{19Mar24} and Corollary \ref{a20Mar24}, there are inclusions in the set $\Den_l(\mS_n, \ga_n)$  apart from $S_y$: 
  \begin{equation}
S_y\subseteq {}'\mY_n  \subseteq \pS (\mS_n)\subseteq \pS (\mS_n)+\ga_n\subseteq \CT,  \;  {}'\mY_n\subseteq \widetilde{{}'\mY_n}\subseteq {}'\mY_n+\ga_n \subseteq \pS (\mY_n)+\ga_n, \; {}'\mY_n+\ga_n \subseteq \CS\subseteq \CT.
\end{equation}

\begin{corollary}\label{a20Mar24}%\marginpar{a20Mar24}
 Let $S=\widetilde{{}'\mY_n}, {}'\mY_n+\ga_n, \pS (\mY_n)+\ga_n$. Then $S\in \Den_l(\mS_n, \ga_n)$ an $S^{-1}\mS_n\simeq \pQ (\mS_n)\simeq K(y_1, \ldots , y_n)$.
 \end{corollary}

\begin{proof} By Theorem \ref{19Mar24}.(2) and Corollary \ref{a19Mar24}.(2), ${}'\mY_n+\ga_n, \pS (\mY_n)+\ga_n\in \Den_l(\mS_n, \ga_n)$. Now,   for $S= {}'\mY_n+\ga_n, \pS (\mY_n)+\ga_n$, the corollary follows from 
 Corollary \ref{a19Mar24}.(2). It remains to consider the case when $S=\widetilde{{}'\mY_n}$.  

(i) {\em The pair $(\widetilde{{}'\mY_n}, {}'\mY_n+\ga_n)$ is a left localization pair of the ring $\mS_n$}:  The inclusions 
${}'\mY_n\subseteq \widetilde{{}'\mY_n}\subseteq {}'\mY_n+\ga_n$ yield 
$$\ga_n=\ass_l({}'\mY_n)\subseteq \ass_l(\widetilde{{}'\mY_n})\subseteq \ass_l({}'\mY_n+\ga_n)=\ga_n.$$
Therefore, $\ass_l(\widetilde{{}'\mY_n})=\ga_n=\ass_l({}'\mY_n+\ga_n)$.  For each element $s\in {}'\mY_n+\ga_n$, there is an element $\alpha\in \N^n$ such that $y^\alpha s\in {}'\mY_n\subseteq \widetilde{{}'\mY_n}$. Notice that $y^\alpha \in S_y \subseteq{}'\mY_n\subseteq \widetilde{{}'\mY_n}$, and the statement (i) follows.

(ii) $\widetilde{{}'\mY_n}\in \Den_l(\mS_n , \ga_n)$  {\rm and} $\widetilde{{}'\mY_n}^{-1}\mS_n\simeq ({}'\mY_n+\ga_n)^{-1}\mS_n\simeq K(y_1, \ldots , y_n)$: The statement (ii) follows from the statement (i) and Lemma \ref{c16Mar15} where     $S=\widetilde{{}'\mY_n}$ and  $T={}'\mY_n+\ga_n\in \Den_l(\mS_n,\ga_n)$.
\end{proof}

\begin{corollary}\label{a22Mar24}%\marginpar{a22Mar24}

\begin{enumerate}
\item  $Q' (\mS_n)\simeq \eta (\pQ(\mS_n))= K(x_1, \ldots , x_n)$.

\item $\CC_{\mS_n}'=S' (\mS_n)$ and $Q_{l,cl}'(\mS_n)=Q'(\mS_n)\simeq K(x_1, \ldots , x_n)$.
\end{enumerate}

 \end{corollary}
 
 \begin{proof} 1. Since $\eta$ is an involution of the algebra $\mS_n$,  $Q' (\mS_n)\simeq \eta (\pQ(\mS_n))$.
 Since $\eta (y_i)=x_i$ for all  $i=1,\ldots , n$,  $\eta (\pQ(\mS_n))= \eta(K(y_1, \ldots , y_n)=K(x_1, \ldots , x_n)$, by Theorem \ref{19Mar24}.(1).
 
 2. By statement 1, the ring $ Q'(\mS_n)\simeq K(x_1, \ldots , x_n)$ is a field. In particular, it is a semisimple Artinian ring, and so statement 2 follows from Theorem \ref{8Mar15}.(2). 
 \end{proof}

{\bf Descriptions of the sets $\pCC_{\mS_1}$ and $\CC_{\mS_1}'$.}

(i) $\pCC_{\mS_1}\subseteq \mS_1\backslash F$: It is obvious that every element of the ideal $F=\oplus_{i,j\in \N}KE_{ij}\simeq M_\infty (K)$ is a left and right zero divisor of the algebra $F$ (without 1) and of $\mS_1$. 

(ii) $\mY_1^0\subseteq \pCC_{\mS_1}$: The ideal $F$  is an essential right ideal of the algebra  $\mS_1$ such that $F_{\mS_1}\simeq (P_1')^{(\N)}$ is a semisimple right $\mS_1$-module (Proposition \ref{B1Apr24}). By Corollary \ref{e15Mar24}.(1), $\pCC_{\mS_1}=\pCC_{P_1'}$ where $P_1'=K[y]$ is the only simple faithful right $\mS_1$-module. Now, the statement (ii) follows from (\ref{Ppac-1}).

(iii) {\em For each nonzero element $d\in \mS_1\backslash F$, $\der^i d \in \mY_1^0$  for some}  $i\in \N$: The statement (iii) follows (\ref{mS1d}).

(iv) {\em For each nonzero element $d\in \mS_1\backslash F$, $\der^i d\in \pCC_{\mS_1}$  for some  $i\in \N$}: The statement (iii) follows from the statements (ii) and  (iii). 

 Then the well-defined map
%\marginpar{mAn-GNd}
\begin{equation}\label{mAn-GNd}
d: \mS_1\backslash F \ra \N , \;\; a\mapsto d(a) :=\min \{ i\in \N \, | \, \der^ia\in \pCC_{\mS_1}\}
\end{equation}
is called the {\em left regularity degree function} and the natural number $d(a)$ is called the {\em left regularity degree} of $a$. For each element $a\in \mS_1\backslash F$, $d(a)$ can be found in finitely many steps.  %where the explicit expression  (\ref{daca})  is given for $d(a)$.
Now, Theorem \ref{A1Apr24}.(1) follows. Then Theorem \ref{A1Apr24}.(2) follows from Theorem \ref{A1Apr24}.(1). 

\begin{theorem}\label{A1Apr24}%\marginpar{A1Apr24}

\begin{enumerate}
\item $\pCC_{\mS_1} = \{ \der^{d(a)}a\, | \, a\in \mS_1\backslash F\}$.
\item $\CC_{\mS_1}'= \eta(\pCC_{\mS_1})$.
\end{enumerate}
\end{theorem}
By (\ref{mS1d}), each element $a\in \mS_1$ is  a unique sum $a=\sum_{i=0}^l\l_{-i}y^i+\sum_{j=1}^m\l_j x^j+a_F$ where $\l_k\in K$, $k=-l, \ldots , m$  and $a_F\in F$. Let $a_y:=\sum_{i=0}^l\l_{-i}y^i$. The integer
 $$s(a_F):=
\begin{cases}
\min \{ n\in \N \, | \, a_F\in \bigoplus_{i,j=0}^nKe_{ij}\} & \text{if }a_F\neq 0,\\
-1& \text{if } a_F=0.\\
\end{cases}  
$$
is called the {\em size} of the element $a_F\in F$. The integer $s(a) := s(a_F)$ is called the {\em size} of the element $a$. For each $i\in \N$, let $P_{1,\leq i}':=\{ a\in P_1'\, | \, \deg_y(a)\leq i\}$ where $\deg_y$ is the degree of the polynomial $a\in P_1'$ in the variable $y$.

 \begin{theorem}\label{A31Mar24}%\marginpar{A31Mar24}

\begin{enumerate}
\item  $\pCC_{\mS_1}=\{ a\in \mS_1\backslash (xK[x]+F)\, | \, \cdot a: P_{1,\leq s(a)}'\ra P_{1,\leq  s(a)+\deg_y(a_y)}'$, $p\mapsto pa$  is an injection$\}$.

\item $\CC'_{\mS_1}=\eta (\pCC_{\mS_1})$ where $\eta$ is the involution of the algebra $\mS_1$, see (\ref{etamSn}).
\end{enumerate}
\end{theorem}  
 
\begin{proof} 1. 
(i) $\pCC_{\mS_1}\cap (xK[x]+F)=\emptyset$: Suppose that  $a\in xK[x]$. Then $1\in \ker(\cdot a)$. 
 Recall that $\pCC_{\mS_1}\subseteq \mS_1\backslash F$ (see the statement (i) in the proof of Theorem \ref{A1Apr24}). So, it remains to consider the case when $a\in (xK[x]+F)\backslash  (xK[x]\cup F)$.  
Then the map $\cdot a: P_{1,\leq s(a)+1}'\ra P_{1,\leq s(a)}'$, $p\mapsto pa$ is a well-defined map. Since 
$$\dim_K(P_{1,\leq s(a)+1}')=s(a)+2>s(a)+1=\dim_K( P_{1,\leq s(a)}'),$$ $\ker(\cdot a)\neq 0$.  Therefore, $\pCC_{\mS_1}\cap (xK[x]+F)=\emptyset$.

(ii) {\em For each  element} $ a\in \mS_1\backslash (xK[x]+F)$, $\ker_{P_1'}(\cdot a)\subseteq  P_{1,\leq s(a)}'$: Since $a\in \mS_1\backslash (xK[x]+F)$, $a_y:=\sum_{i=0}^l\l_{-i}y^i\neq 0$ where $\l_{-i}\in K$. Suppose that $\l_{-l}\neq 0$. Suppose that $p\in \ker_{P_1'}(\cdot a)\backslash  P_{1,\leq s(a)}'$, i.e. $\deg_y(p)>s(a)$. Then $$\deg(pa)=l+\deg_y(p),$$ a contradiction (since $pa=0$).

Now, statement 1 follows from statements (i) and (ii).
 
2. Statement 2 follows from statement 1.
\end{proof}

{\bf The algebras $\CI_n$ of scalar integro-differential operators.} In the next section, we will see that 
 the algebra $\mI_n$ of polynomial  integro-differential operators contains the algebra of scalar integro-differential operators, \cite{algintdif}: 
$$
\CI_n:= K\bigg\langle
 \der_1, \ldots ,\der_n,  \int_1,
\ldots , \int_n\bigg\rangle .
$$ 
The  algebra $\CI_n $ is canonically isomorphic to the algebra $\mS_n$, \cite[Eq. (9)]{algintdif}, see 
(\ref{SnIniso}) for an explicit isomorphism.

\begin{corollary}\label{a21Mar24}%\marginpar{a21Mar24}
\begin{enumerate}

\item  $\pCC_{\CI_n}=\pS (\CI_n)$ and  ${}'Q_{l,cl}(\CI_n)=\pQ (\CI_n)\simeq \pQ (\mS_n)\simeq K(y_1, \ldots , y_n)$.

\item  $\CC_{\CI_n}=S'(\CI_n)$ and $Q'_{l,cl} (\CI_n)=Q' (\CI_n)\simeq Q' (\mS_n)\simeq K(x_1, \ldots , x_n)$.
\end{enumerate}
\end{corollary}
 
 \begin{proof} 1.  The algebras $\CI_n$ and $\mS_n$ are isomorphic (see (\ref{SnIniso})) and statement 1 follows from Theorem \ref{19Mar24} and Corollary \ref{a31Mar24}.
 
 2. Statement 2 follows from Corollary \ref{a22Mar24}. 
 \end{proof}

\begin{corollary}\label{b21Mar24}%\marginpar{b21Mar24}

 $\pCC_{\CI_1}=\xi (\pCC_{\mS_1})$ and  $\CC_{\CI_1}'=\xi (\CC_{\mS_1}')$.

\end{corollary}

%%%%%%%%%%%%%%%%%% SECTION 4 %%%%%%%%%%%%%%%%%%%%%

\section{The rings $\pQ(\mI_n)$ and  $Q'(\mI_n)$} \label{PQmIn-PQCIn} %\marginpar{PQmIn-PQCIn}

The aim of the section is to prove Theorem \ref{18Mar24},  to obtain a description of the sets $\pCC_{\mI_1}$ and $\CC_{\mI_1}'$ (Theorem \ref{C31Mar24}), and to prove Theorem \ref{A25Mar24} which is a criterion for  $\pQ (\mI_n)\simeq Q(I_n)$. \\

{\bf The rings $\mI_n$ of integro-differential operators and the Jacobian algebras $\mA_n$.}
%Throughout, ring means an associative ring with $1$; module means
%a left module;
% $\N :=\{0, 1, \ldots \}$ is the set of natural numbers; 
In this section the following notation is fixed: $K$ is a
field of characteristic zero and  $K^*$ is its group of units;
$P_n:= K[x_1, \ldots , x_n]$ is a polynomial algebra over $K$;
$\der_1:=\frac{\der}{\der x_1}, \ldots , \der_n:=\frac{\der}{\der
x_n}$ are the partial derivatives ($K$-linear derivations) of
$P_n$; $\End_K(P_n)$ is the algebra of all $K$-linear maps from
$P_n$ to $P_n$;
%and $\Aut_K(P_n)$ is its group of units (i.e. the
%group of all the invertible linear maps from $P_n$ to $P_n$);
the
subalgebra  $A_n:= K \langle x_1, \ldots , x_n , \der_1, \ldots ,
\der_n\rangle$ of $\End_K(P_n)$ is called the $n$'th {\bf  Weyl
algebra}.

\begin{definition} (\cite{jacalg}) The {\bf Jacobian algebra}
$\mA_n$ is the subalgebra of $\End_K(P_n)$ generated by the Weyl
algebra $A_n$ and the elements $H_1^{-1}, \ldots , H_n^{-1}\in
\End_K(P_n)$ where $$H_1:= \der_1x_1, \ldots , H_n:= \der_nx_n.$$
\end{definition}

Clearly, $\mA_n =\bigotimes_{i=1}^n \mA_1(i) \simeq \mA_1^{\t n }$
where $\mA_1(i) := K\langle x_i, \der_i , H_i^{-1}\rangle \simeq
\mA_1$. The algebra $\mA_n$ contains all the  integrations
$\int_i: P_n\ra P_n$, $ p\mapsto \int p \, dx_i$, i.e.  $$
\int_i= x_iH_i^{-1}: x^\alpha \mapsto (\alpha_i+1)^{-1}x_ix^\alpha.
$$ 
The algebra $\mA_n$ contains the {\bf algebra of  polynomial integro-differential
operators}, \cite{algintdif}:
$$\mI_n:=K\bigg\langle x_1, \ldots , x_n,
 \der_1, \ldots ,\der_n,  \int_1,
\ldots , \int_n\bigg\rangle .$$ 
Notice that $\mI_n=\bigotimes_{i=1}^n\mI_1(i)\simeq
\mI_1^{\t n}$ where $\mI_1(i):= K\langle x_i, \der_i,
\int_i\rangle\simeq \mI_1$. The algebra $\mI_n$ contains the {\bf algebra of scalar integro-differential operators}, \cite{algintdif}: 
$$
\CI_n:= K\bigg\langle
 \der_1, \ldots ,\der_n,  \int_1,
\ldots , \int_n\bigg\rangle .
$$ 
The  algebra $\CI_n $ is canonically isomorphic to the algebra $\mS_n$, \cite[Eq. (9)]{algintdif}:
%\marginpar{SnIniso}
\begin{equation}\label{SnIniso}
\xi: \mS_n\ra \CI_n, \;\; x_i\mapsto \int_i, \;\; y_i\mapsto
 \der_i, \;\; i=1, \ldots , n.
\end{equation}

For the reader's convenience we collect some
known results on the algebras $\mI_n$ and $\mA_n$ from the papers \cite{jacalg, %jacaut, 
algintdif} that are used later in
the paper.  The algebra $\mI_n$ is a prime, central, catenary, non-Noetherian
algebra of classical Krull dimension $n$ and of Gelfand-Kirillov
dimension $2n$, \cite{algintdif}.
 Since $x_i= \int_iH_i$, where $H_i:=\der_ix_i$, the
algebra $\mI_n$ is generated by the elements $\{ \der_i , H_i,
\int_i\, | \, i=1, \ldots , n\}$, and $\mI_n =\bigotimes_{i=1}^n
\mI_1(i)$ where
$$\mI_1(i):= K\bigg\langle \der_i , H_i,
\int_i\bigg\rangle=K\bigg\langle \der_i , x_i, \int_i\bigg\rangle\simeq \mI_1.$$

 When $n=1$ we usually drop the subscript `1' in $\der_1$, $\int_1$, $H_1$, and $x_1$.
 The algebra $\mI_1 = \oplus_{i \in \Z} \mI_{1, i}$ is
a $\Z^n$-graded algebra where 
%\marginpar{mIn-A1i}
\begin{equation}\label{mIn-A1i}
\mI_{1,i}=\begin{cases}
\int^iD_1& \text{if $i\geq 1$},\\
D_1& \text{if $i=0$},\\
D_1 \der^{-i}& \text{if $i\leq -1$},\\
\end{cases}
\end{equation}
where $D_1:= K[H]\oplus\bigoplus_{i\in \N}Ke_{ii}$ is a commutative, not Noetherian,  not finitely generated algebra and $K[H]$ is a polynomial algebra in the variable $H$ and $H=\der x=x\der+1$.

 The following elements of the algebra $\mI_1=K\langle
\der , H, \int \rangle$,
 %\marginpar{eijdef}
\begin{equation}\label{eijdef}
e_{ij}:=\int^i\der^j-\int^{i+1}\der^{j+1}, \;\; i,j\in \N ,
\end{equation}
satisfy the relations: $e_{ij}e_{kl}=\d_{jk}e_{il}$ where
$\d_{ij}$ is the Kronecker delta. The matrices of the linear maps
$e_{ij}\in \End_K(K[x])$  with respect to the basis $\{ x^{[s]}:=
\frac{x^s}{s!}\}_{s\in \N}$ of the polynomial algebra $K[x]$  are
the elementary matrices, i.e.
$$ e_{ij}*x^{[s]}=\begin{cases}
x^{[i]}& \text{if }j=s,\\
0& \text{if }j\neq s.\\
\end{cases}$$
The direct sum  $F:=\bigoplus_{i,j\in \N}Ke_{ij}$ is the only
proper (hence maximal) ideal of the algebra $\mI_1$. As an algebra
without 1 it is isomorphic to the algebra without 1 of infinite
dimensional matrices $M_\infty (K) :=\varinjlim
M_d(K)=\bigoplus_{i,j\in \N}KE_{ij}$ via $e_{ij}\mapsto E_{ij}$
where $E_{ij}$ are the matrix units. For all $i,j\in \N$,
%\marginpar{Ieij}
\begin{equation}\label{Ieij}
\int e_{ij}=e_{i+1,j}, \;\; e_{ij}\int = e_{i,j-1}, \;\;\der
e_{ij}=e_{i-1,j}, \;\; e_{ij}\der= e_{i,j+1},
\end{equation}
where $e_{-1, j}:=0$ and $e_{i,-1}:=0$. 
%\marginpar{I1dis}
\begin{equation}\label{I1dis}
 \mI_1=\bigoplus_{i\geq 1}K[H]\der^i\oplus K[H]\oplus \bigoplus_{i\in \geq 1}K[H]\int^i \oplus F
\end{equation}
and $K[H]\der^i=\der^iK[H]$ and $K[H]\int^i=\int^i K[H]$ for all $i\geq 1$. 
 The algebra $\mI_1$ is generated by the elements $ \der$,  $\int$ and $ H$ subject to  the following
defining relations (Proposition 2.2, \cite{algintdif}):
$$\der\int= 1, \;\; \bigg[H, \int\bigg]=\int, \;\;
[ H, \der]=-\der, \;\; H \bigg(1-\int\der\bigg) = \bigg(1-\int\der\bigg) H=1-\int\der.$$

 The algebra $\mI_n=\otimes_{i=1}^n\mI_1(i) = \oplus_{\alpha \in \Z^n} \mI_{n, \alpha}$ is
a $\Z^n$-graded algebra where $\mI_{n,\alpha}:=\t_{k=1}^n
\mI_{1,\alpha_k}(k)$ for $\alpha=(\alpha_1, \ldots , \alpha_n)$.
 The algebra
$\mI_n$ contains the ideal $$F_n:=
F^{\t n}=\bigotimes_{i=1}^nF(i)=\bigoplus_{\alpha, \beta \in \N^n}
Ke_{\alpha\beta}$$ where $e_{\alpha\beta}:= \prod_{i=1}^n
e_{\alpha_i\beta_i}(i)$,
$e_{\alpha_i\beta_i}(i):=\int_i^{\alpha_i}\der_i^{\beta_i}-
\int_i^{\alpha_i+1}\der_i^{\beta_i+1}$ and $F(i)=\bigoplus_{s,t\in
\N}Ke_{st}(i)$.

\begin{lemma}\label{b10Oct9}%\marginpar{b10Oct9}

\begin{enumerate}

\item {\em (\cite[Corollary 3.3.(2)]{algintdif})} The set of height one prime ideals of the algebra $\mI_n$ is
$\{ \gp_1:= F\t\mI_{n-1}, \gp_1:= \mI_1\t F\t\mI_{n-2},\ldots ,
\gp_n:= \mI_{n-1}\t F\}$. 

\item {\em (\cite[Corollary 3.3.(3)]{algintdif})} Each ideal of the algebra $\mI_n$
is an idempotent ideal ($\ga^2= \ga$). 

%\item The ideals of the
%algebra $\mI_n$ commute ($\ga \gb = \gb \ga$). \item The lattice
%$\CJ (\mI_n)$ of ideals of the algebra $\mI_n$ is %distributive.

\item {\rm (\cite[Lemma 5.2.(2)]{algintdif})}  Each nonzero ideal
of the algebra $\mI_n$ is an essential left and right submodule of
$\mI_n$.

% \item $\ga \gb = \ga \cap \gb$ for all ideals $\ga $ and
%$\gb $ of the algebra $\mI_n$. 

\item {\em (\cite[Corollary 3.3.(8)]{algintdif})} The ideal $\ga_n :=
\gp_1+\cdots + \gp_n$ is the largest (i.e. the only maximal)
ideal of $\mI_n$  and $F_n = F^{\t
n}=\cap_{i=1}^n \gp_i$ is the smallest nonzero ideal of $\mI_n$.

%\item {\rm (A classification of ideals of $\mI_n$)} The map
% $\CC_n\ra \CJ (\mI_n)$, $C\mapsto I_C:= \sum_{f\in C}I_f$
% is a bijection where $I_\emptyset :=0$. The number of ideals of
% $\mI_n$ is equal to the Dedekind  number $\gd_n$.
% In
%particular, there are only finitely many ideals, say $s_n$,  of
%$\mI_n$. Moreover, $2-n+\sum_{i=1}^n2^{n\choose i}\leq s_n \leq
%2^{2^n}$.
% For $n=1$, $F$ is the unique proper ideal of the algebra
%$\mI_1$. 

% \item {\rm (A classification of prime ideals of
%$\mI_n$)} Let $\Sub_n$ be the set of all subsets of $\{ 1, \ldots
%, n\}$. The map $\Sub_n\ra \Spec (\mI_n)$, $ I\mapsto \gp_I:=
%\sum_{i\in I}\gp_i$, $\emptyset \mapsto 0$, is a bijection, i.e.
%any nonzero prime ideal of $\mI_n$ is a unique sum of primes of
%height 1; $|\Spec (\mI_n)|=2^n$; the height of $\gp_I$ is $| I|$;
%and

% \item  $\gp_I\subset \gp_J$ iff $I\subset
% J$.

\item  {\rm (\cite[Proposition  3.8]{algintdif})} 
 The polynomial algebra $P_n$ is the only (up to isomorphism) faithful simple left $\mI_n$-module and ${}_{\mI_n}P_n\simeq \mI_n/\mI_n(\der_1, \ldots , \der_n)$ {\rm (\cite[Proposition  3.4.(3)]{intdifaut})}).
 
 \item  {\rm (\cite[Lemma  5.2.(1)]{algintdif})} For all nonzero ideals $\ga$ of the algebra $\mI_n$,
$\lann_{\mI_n}(\ga ) = \rann_{\mI_n}(\ga) =0$. \item   {\rm (\cite[Lemma  5.2.(2)]{algintdif})} Each nonzero
ideal of the algebra $\mI_n$ is an essential left and right
submodule of $\mI_n$.
 
\end{enumerate}
\end{lemma}

{\bf The involution $*$ on the algebra $\mI_n$}. The algebra $\mI_n$ admits the involution: 
%\marginpar{*invIn}
\begin{equation}\label{*invIn}
*: \mI_n\ra \mI_n, \;\; \der_i \mapsto  \int_i, \;\; \int_i\mapsto
\der_i, \;\; H_i\mapsto H_i, \;\; i=1, \ldots , n,
\end{equation}
i.e. it is a $K$-algebra {\em anti-isomorphism} $((ab)^*= b^*a^*)$
such that $*\circ *= \id_{\mI_n}$. Therefore, the algebra $\mI_n$
is {\em self-dual}, i.e. is isomorphic to its {\em opposite}
algebra $\mI_n^{op}$. As a result, the left and the right
properties of the algebra $\mI_n$ are the same.  For all elements
$\alpha , \beta \in \N^n$, 
%\marginpar{eab*}
\begin{equation}\label{eab*}
e_{\alpha\beta}^*= e_{\beta \alpha}.
\end{equation}

The involution $*$ can be extended to an involution of the algebra
$\mA_n$ by setting
$$ x_i^*=H_i\der_i, \;\; \der_i^* = \int_i, \;\; (H_i^{\pm 1})^*=
H_i^{\pm 1}, \;\; i=1, \ldots , n.$$  Note that $y_i^* =
(H_i^{-1}\der_i)^* = \int_i H_i^{-1} = x_iH_i^{-2}$, $A_n^*
\not\subseteq A_n$,  but $\CI_n^* = \CI_n$ where $$\CI_n:=
K\bigg\langle \der_1, \ldots , \der_n \int_1, \ldots , \int_n\bigg\rangle $$
is the algebra of integro-differential operators with constant
coefficients.

For a subset $S$ of a ring $R$, the sets $\lann_R(S):= \{ r\in R\,
| \, rS=0 \}$ and $\rann_R(S):= \{ r\in R\, | \, Sr=0\}$ are
called the {\em left} and the {\em right annihilators} of the set
$S$ in $R$.
 Using the fact that the algebra $\mI_n$ is a GWA and its
 $\Z^n$-grading, we see that
 %\marginpar{laI}
\begin{equation}\label{laI}
\lann_{\mI_n}\bigg(\int_i\bigg) = \bigoplus_{k\in \N}Ke_{k0}(i)\bigotimes
\bigotimes_{i\neq j}\mI_1(j), \;\; \rann_{\mI_n}\bigg(\int_i\bigg) =0.
\end{equation}
 %\marginpar{lad}
\begin{equation}\label{lad}
\rann_{\mI_n}(\der_i) = \bigoplus_{k\in \N}Ke_{0k}(i)\bigotimes
\bigotimes_{i\neq j}\mI_1(j), \;\; \lann_{\mI_n}(\der_i) =0.
\end{equation}

Let $\ga$ be an ideal of the algebra $\mI_n$. The factor algebra
$\mI_n / \ga$ is a Noetherian algebra iff $\ga = \ga_n$
(Proposition 4.1,  \cite{algintdif}). The factor algebra $B_n :=
\mI_n / \ga_n$ is isomorphic to the skew Laurent polynomial
algebra
$$\bigotimes_{i=1}^n K[H_i][\der_i,
\der_i^{-1}; \tau_i]=\CP_n [ \der_1^{\pm 1}, \ldots, \der_n^{\pm
1}; \tau_1, \ldots , \tau_n],$$ via $\der_i\mapsto \der_i$,
$\int_i\mapsto \der_i^{-1}$, $H_1\mapsto H_i$  (and $x_i\mapsto
\der_i^{-1}H_i$) where $\CP_n := K[H_1, \ldots , H_n]$ and
$\tau_i(H_i) = H_i+1$. We identify these two algebras via this
isomorphism. It is obvious that
$$ B_n =\bigotimes_{i=1}^n K[H_i][z_i,
z_i^{-1}; \s_i]=\CP_n [ z_1^{\pm 1}, \ldots, z_n^{\pm 1}; \s_1,
\ldots , \s_n],$$ where $z_i:=\der_i^{-1}$ and  $\s_i =
\tau_i^{-1}: H_i\mapsto H_i-1$. The algebra $B_n$ is also
the left (but not right) localization of the algebra $\mI_n$ at
the multiplicatively closed set $$S_\der:=S_{\der_1, \ldots , \der_n}:=\{
\der_1^{\alpha_1}\cdots \der_n^{\alpha_n}\, | \, (\alpha_i)\in
\N^n\}, \;\; {\rm and}\;\; B_n \simeq S_\der^{-1}\mI_n.$$
 The algebra $B_n$ contains the algebra $\D_n:= \CP_n[\der_1, \ldots , \der_n;\tau_1,\ldots ,\tau_n]$ which is a skew polynomial ring. \\

Using the involution on the algebra 
$*$ on $\mI_n$,  the polynomial algebra $P_n$
can be seen as the {\em right} $\mI_n$-module by the rule 
$$pa:= a^* p\;\; {\rm  for\;
all}\;\;p\in P_n\;\; {\rm  and}\;\;a\in \mI_n.$$ By Lemma \ref{b10Oct9}.(5),  $P_n=K[x_1, \ldots , x_n]$  is the only
faithful, simple, right $\mA_n$-module. Let $P_n':=(P_n)^*=K[\der_1, \ldots, \der_n]$, a polynomial algebra in $n$ variables. Clearly, 

$$(P_n)_{\mA_n}\simeq \big(  \mI_n/\mI_n(H_1-1, \ldots , H_n-1,\der_1, \ldots , \der_n)\Big)^*= \mI_n/\bigg(H_1-1, \ldots , H_n-1,\int_1, \ldots , \int_n\bigg)\mI_n\simeq P_n'\widetilde{1}\simeq (P_n')_{P_n'}$$
where $\widetilde{1}:=1+\bigg(\int_1, \ldots , \int_n\bigg)\mI_n$. So, 
 $P_n'$  {\em is the only
faithful, simple, right $\mI_n$-module}.

\begin{lemma}\label{c29Mar24}%\marginpar{c29Mar24}

\begin{enumerate}

\item ${}_{\mI_n}F_n\simeq P_n^{(\N^n)}$ is a direct sum of $N^n$ copies of the simple faithful left $\mI_n$-module $P_n$.

\item $(F_n)_{\mI_n}\simeq (P_n')^{(\N^n)}$  is a direct sum of $N^n$ copies of the simple faithful right $\mI_n$-module $P_n'$.

\end{enumerate}
\end{lemma}

\begin{proof} 1.  ${}_{\mI_n}F_n=\oplus_{\alpha, \beta \in \N^n}KE_{\alpha\beta}= \oplus_{\beta \in \N^n}\Big(\oplus_{\alpha,  \in \N^n}KE_{\alpha\beta} \Big) \simeq \oplus_{\beta \in \N^n}P_n\simeq P_n^{(\N^n)}$.

2. Statement 2 follows from statement 1 by applying the involution $*$ and using the fact that $F_n^*=F_n$.\end{proof}

{\bf The ring $\pQ (\mI_n)$.} Recall that $\D_n=\CP_n[\der_1, \ldots , \der_n; \tau_1, \ldots , \tau_n]$. Let $\D_n^0:=\D_n\backslash \{ 0\}$, ${}'\D_n^0:=\D_n\cap \pCC_{\mI_n}$, 
$${}'\D_n:=\D_n\cap \pS (\mI_n)\;\; {\rm  and }\;\;\widetilde{{}'\D_n}:=\{c\in \D_n\, | \, \der^\alpha c\in \pS (\mI_n)\; {\rm  for\; some}\;\;\alpha \in \N^n\}.
$$
 Notice that  ${}'\D_n\subseteq 
{}'\D_n^0\subseteq \D_n^0$. Theorem \ref{18Mar24} produces explicit left denominators sets   $S\in \Den_l(\mI_n, \ga_n)$ such that  $S^{-1}\mI_n\simeq \pQ (\mI_n)$. 
  By Theorem \ref{18Mar24}, there are inclusions in the set $\Den_l(\mI_n, \ga_n)$ apart from $S_\der$: 
\begin{equation}
S_\der\subseteq {}'\D_n  \subseteq \pS (\mI_n)\subseteq \pS (\mI_n)+\ga_n\;\; {\rm and}\;\;{}'\D_n\subseteq \widetilde{{}'\D_n}\subseteq '\D_n+\ga_n \subseteq \pS (\mI_n)+\ga_n.
\end{equation}

\begin{theorem}\label{18Mar24}%\marginpar{18Mar24}
 
\begin{enumerate}

\item ${}'\D_n\in \Den_l(\mI_n, \ga_n)$ and ${}'\D_n^{-1}\mI_n\simeq \pQ (\mI_n)$. Furthermore, the  subset  ${}'\D_n$ of $\mI_n$  is a left denominator set of $\mI_n$ which is the  largest  left denominator set that is contained in the multiplicative set $\pCC_{\mI_n}\cap \D_n$.

\item $  {}'\D_n+\ga_n, \pS (\mI_n)+\ga_n \in \Den_l(\mI_n, \ga_n)$ and 
$$ 
 \Big({}'\D_n+\ga_n\Big)^{-1}\mI_n\simeq \Big( \pS (\mI_n)+\ga_n\Big)^{-1}\mI_n\simeq \pQ (\mI_n).  
$$

\item ${}'\D_n=\overline{{}'\D_n},\overline{{}'\D_n+\ga_n}, \overline{\pS (\mI_n)+\ga_n} \in \Den_l(L_n, 0)$ and 
$$ 
{}'\D_n^{-1}L_n\simeq \overline{{}'\D_n+\ga_n}^{-1}L_n\simeq \overline{\pS (\mI_n)+\ga_n}^{-1}L_n\simeq  \pQ (\mI_n)  
$$
where $\bS:= \pi_{\ga_n}(S)$ and  $\pi_{\ga_n}: \mI_n\ra L_n=\mI_n/\ga_n$, $r\mapsto \br := r+\ga_n$.

\item $\widetilde{{}'\D_n}\in \Den_l(\mI_n, \ga_n)$ and $\widetilde{{}'\D_n}^{-1}\mI_n\simeq \pQ (\mI_n)$. $\overline{\widetilde{{}'\D_n}}\in \Den_l(L_n, 0)$ and $\overline{\widetilde{{}'\D_n}}^{-1}L_n\simeq \pQ (\mI_n)$.

\item $\widetilde{{}'\D_n}+\ga_n\in \Den_l(\mI_n, \ga_n)$ and $\Big(\widetilde{{}'\D_n}+\ga_n\Big)^{-1}\mI_n\simeq \pQ (\mI_n)$.

\item ${}'\D_n \in \Den_l(\D_n, 0)$ and 
$
{}'\D_n^{-1}\D_n\simeq  \pQ (\mI_n)$.

\item ${}''\D_n:=\pS (\mI_n)\cap (\pCC_{\D_n+\ga_n}+\ga_n)+\ga_n\in \Den_l(\mI_n, \ga_n)$, $
{}'\D_n\subseteq 
{}''\D_n$ and $
{}''\D_n^{-1}\mI_n\simeq  \pQ (\mI_n)$.
\end{enumerate}
\end{theorem}

\begin{proof} 1. (i)  ${}'\D_n\in \Den_l(\mI_n, \ga_n)$ {\em and} ${}'\D_n^{-1}\mI_n\simeq\pQ (\mI_n)$: By the definition, the set    ${}'\D_n$ is a multiplicative subset of $\pS (\mI_n)\subseteq \mI_n$. It follows from the inclusions $S_\der \subseteq {}'\D_n\subseteq \pS (\mI_n)$ that 
$$\ga_n = \ass_l(S_\der )\subseteq \ass_l({}'\D_n)\subseteq \ass_l(\pS (\mI_n))=\ga_n, $$
and so $\ass_l({}'\D_n)=\ga_n=\ass_l(\pS (\mI_n))$. For each element $s\in \pS (\mI_n)$, there is an element $\alpha\in \N^n$ such that $\der^\alpha s\in {}'\D_n$. Notice that $\der^\alpha \in S_\der \subseteq {}'\D_n$. 
Now, the statement (i) follows from Lemma \ref{c16Mar15} where $T=\pS (\mI_n)\in \Den_l(\mI_n,\ga_n)$ and  $S={}'\D_n$.

(ii) {\em The  subset  ${}'\D_n$ of $\mI_n$  is a left denominator set of $\mI_n$ which is the  largest  left denominator set that is contained in the multiplicative set $\pCC_{\mI_n}\cap \D_n$}: Let $T$ be a left denominator set of $\mI_n$ which is the  largest  left denominator set that is contained in the multiplicative set $\pCC_{\mI_n}\cap \D_n$. By the statement (i),  ${}'\D_n\in \Den_l(\mI_n, \ga_n)$. Clearly,  ${}'\D_n\subseteq  \pCC_{\mI_n}\cap \D_n$ and so ${}'\D_n\subseteq T$. Since $T\subseteq \pCC_{\mI_n}$ and $\pS (\mI_n)$ is a the largest left denominator set in $\pCC_{\mI_n}$, we have the inclusion $T\subseteq \D_n\cap\pS(\mI_n)={}'\D_n$. Therefore, $T={}'\D_n$.

2. By statement 1, ${}'\D_n, \pS (\mI_n)\in\Den_l(\mI_n, \ga_n)$ and ${}'\D_n^{-1}\mI_n\simeq  \pS (\mI_n)^{-1}\mI_n=\pQ (\mI_n)$. Now, statement 2 follows from Corollary \ref{a19Mar24}.(2).   

3.  Statement 3 follows at once from statement 2 (If $S\in \Den_l(R,\ga)$ then $\bS :=S+\ga \in \Den_l(\bR, 0)$ and $S^{-1}R\simeq \bS^{-1}\bR$ where $\bR := R/\ga$).

4. By the definition, the set $\widetilde{{}'\D_n}$ is a multiplicative set such that ${}'\D_n\subseteq \widetilde{{}'\D_n}\subseteq \pS (\mI_n)$. Therefore,  $\ga_n=\ass_l({}'\D_n)\subseteq \ass_l(\widetilde{{}'\D_n})\subseteq \ass_l(\pS (\mI_n))=\ga_n$, and so $\ass_l(\widetilde{{}'\D_n})=\ga_n$. Now, the first part of statement 4 follows from the definition of the set $\widetilde{{}'\D_n}$ and Lemma \ref{c16Mar15} where $S=\widetilde{{}'\D_n}$ and $T=\pS(\mI_n)$.

The second  part of statement 4 follows from the first one.

5. Statement 5 follows from statement 4.

6. By statement 3,  ${}'\D_n \in \Den_l(L_n, 0)$ and $ {}'\D_n^{-1}L_n\simeq  \pQ (\mI_n)$.  Now, statement 6 follows from  Lemma \ref{b18Mar24}.(2) where $R=\D_n$,  $R'=L_n$ and  $S={}'\D_n$ (The $R$-module $R'/R$ is $S_\der$-torsion. Hence it is also ${}'\D_n$-torsion as $S_\der\subseteq {}'\D_n$).

7. By the definition, the subset ${}''\D_n$ of $\mI_n$ is a multiplicative set such that ${}'\D_n\subseteq {}''\D_n\subseteq \pS(\mI_n)$, see statements 1 and 2. Hence, 
$\ga_n=\ass_l({}'\D_n)\subseteq \ass_l({}''\D_n)\subseteq \ass_l(\pS(\mI_n))=\ga_n$, and so $\ass_l({}''\D_n)=\ga_n=\ass_l(\pS(\mI_n))$. Notice that $S_\der\subseteq {}'\D_n\subseteq {}''\D_n$ and for each element $s\in {}''\D_n$ there is an element $\alpha\in \N^n$ such that $\der^\alpha s\in {}'\D_n\subseteq {}''\D_n$. Now, statement 7 follows from 
 Lemma \ref{c16Mar15}  where $S={}''\D_n$ and  $T=\pS(\mI_n)$.
\end{proof}

In order to prove Theorem \ref{A25Mar24},  we need the following lemma which is a characterization of the set $\pCC_{\mI_n}$.

\begin{lemma}\label{In-b29Mar24}%\marginpar{In-b29Mar24}
Let $a\in \mI_n$. Then the following statements are equivalent:
\begin{enumerate}

\item $a\in \pCC_{\mI_n}$.

\item $a\in {}'\CC_{F_n}$.

\item  $a\in {}'\CC_{P_n'}$ where $P_n'=K[\der_1, \ldots,  \der_n]$ is the only simple, faithful, right $\mI_n$-module.

\item $a^*\in \CC_{P_n}'$  where $P_n=K[x_1, \ldots  ,x_n]$ is the only simple, faithful, left $\mI_n$-module.

\end{enumerate}
\end{lemma}

\begin{proof} $(1\Leftrightarrow 2)$ The equivalence follows from Corollary \ref{d15Mar24}.(1) and  the fact that every nonzero ideal is an essential right ideal of the algebra $\mI_n$ (Lemma \ref{b10Oct9}.(7)).

 $(2\Leftrightarrow 3)$ The equivalence follows from Corollary \ref{e15Mar24}.(1) and Lemma \ref{c29Mar24}.(1).

 $(3\Leftrightarrow 4)$ The equivalence follows from the fact that $P_n'=P_n^*$.
 \end{proof}
 
 By applying the involution $*$ to Lemma \ref{In-b29Mar24}, we obtain a similar characterization of right regular elements of the ring $\mI_n$ ($a\in \pCC_{\mI_n}$ iff $a^*\in \CC_{\mI_n}'$).\\

{\bf Criterion for  $\pQ (\mI_n)\simeq Q(A_n)$.} The algebras $\D_n$ and  $B_n$  are Noetherian domains. By Goldie's Theorem, their quotient rings are division rings. 
It follows from the inclusions $A_n\subseteq B_n\subseteq Q(A_n)$ and  $\D_n\subseteq S_\der^{-1}\D_n\simeq B_n\subseteq Q(A_n)$ that 
%\marginpar{AQns}
\begin{equation}\label{AQns}
Q(\D_n)= Q(B_n)= Q(A_n).
\end{equation}

\begin{theorem}\label{A25Mar24}%\marginpar{A25Mar24}
The following statements are equivalent:
\begin{enumerate}

\item $\pQ (\mI_n)\simeq Q(A_n)$.

\item $\pQ_{l,cl} (\mI_n)\simeq Q(A_n)$.

\item The set $\overline{\pS (\mI_n)}$ is dense in  $B_n\backslash \{0\}$.

\item The set $\overline{\pCC_{\mI_n}}$ is dense in  $B_n\backslash \{0\}$.

\item The set ${}'\D_n$ is dense in $B_n\backslash \{0\}$.

\item The set ${}'\D_n$ is left  dense in $\D_n^0$.

\item For each element $s\in \D_n^0$, there is an element $s'\in \D_n^0$ such that $s's\in  \pCC_{F_n}$.

\item For each element $s\in \D_n^0$, there is an element $s'\in \D_n^0$ such that $s's\in \pCC_{P_n'}$ where $P_n'=K[\der_1, \ldots , \der_n]$ is the unique simple faithful right $\mI_n$-module.

\end{enumerate}
\end{theorem}

\begin{proof} $(1 \Leftrightarrow 2 \Leftrightarrow 3 \Leftrightarrow 4)$ The equivalence  $(1 \Leftrightarrow 2)$ follows from Theorem \ref{8Mar15}.(2). The equivalences  $(1 \Leftrightarrow 3)$  and  $(2\Leftrightarrow 4)$follows from Theorem \ref{28Feb15} and Theorem \ref{8Mar15}.(2) (since $\ass_l(\pS (\mI_n))=\ga_n$ is a prime ideal of $\mI_n$, the algebra $\mI_n/\ga_n=B_n$ is a domain  and $Q(\mI_n/\ga_n)=Q(B_n)$ is a division ring).

$(1 \Leftrightarrow 5)$ By Theorem \ref{18Mar24}.(1),  ${}'\D_n\in  \Den_l(\mI_n, \ga_n)$ and ${}'\D_n^{-1}\mI_n\simeq \pQ (\mI_n)$.
 Since the set ${}'\D_n$ is dense in  $\pS (\mI_n)$, the equivalence $(1 \Leftrightarrow 5)$ holds iff the equivalence $(1 \Leftrightarrow 3)$ holds.
 
$(5 \Leftrightarrow 6)$ Recall that ${}'\D_n, \D_n^0\in\Den_l (\D_n,0)$, ${}'\D_n\subseteq \D_n^0$  and  $Q(\D_n)=Q(A_n)$, see (\ref{AQns}). By \cite[Lemma 3.5.(3)]{Clas-lreg-quot}, ${}'\D_n^{-1}\D_n\simeq (\D_n^0)^{-1}\D_n=Q(A_n)$ iff
 the set ${}'\D_n$ is left  dense in $\D_n^0$.

$(6 \Leftrightarrow 7)$ By  Lemma \ref{In-b29Mar24}, the inclusion  $s's\in  \pCC_{F_n}$ is equivalent to the inclusion $s's\in  \pCC_{\mI_n}$. Hence, statement 7 is equivalent to the statement that 
  for each element $s\in \D_n^0$, there is an element $s'\in \D_n^0$ such that $s's\in   \pCC_{\mI_n}$, i.e. $s's\in  \D_n\cap  \pCC_{\mI_n}={}'\D_n$, i.e. it is equivalent to statement 6.

$(7\Leftrightarrow 8)$ The equivalence follows from  
 Lemma \ref{In-b29Mar24}.
 \end{proof}
 
{\bf Description of the set $\pCC_{\mA_1}$.}
 By (\ref{I1dis}), each element $a\in \mI_1$ is  a unique sum $a=\sum_{i=0}^ld_{-i}\der^i+\sum_{j=1}^m \int^jd_j+a_F$ where $d_k\in K[H]$, $k=-l, \ldots , m$  and $a_F\in F$. Let $a_\der:=\sum_{i=0}^ld_{-i}\der^i$. The integer
 $$s(a_F):=
\begin{cases}
\min \{ n\in \N \, | \, a_F\in \bigoplus_{i,j=0}^nKe_{ij}\} & \text{if }a_F\neq 0,\\
-1& \text{if } a_F=0.\\
\end{cases}  
$$
is called the {\em size} of the element $a_F\in F$. The integer $s(a) := s(a_F)$ is called the {\em size} of the element $a$. For each $i\in \N$, let $P_{1,\leq i}':=\{ a\in P_1'\, | \, \deg_y(a)\leq i\}$ where $\deg_y\der$ is the degree of the polynomial $a\in P_1'=K[\der]$ in the variable $\der$. 

For all polynomials $p\in K[H]$, $\der p=\tau (p)\der$ where $\tau \in \Aut_K(K[H])$ and  $\tau (H)=H+1$.
 For each nonzero polynomials $p\in K[H]$,
 let 
 $$\mu (p):=\min\{ i\in \N\, | \,{\rm the \; polynomial }\;\; \tau^i(p)\in K[H]\;\; {\rm has\; no\; root\;in}\;\; \N_+\}.$$
 Let $\Psi := \bigoplus_{i\geq 1}\int^iK[H]\oplus F$. 
If $a=\sum_{i=0}^nd_{-i}\der^i+\sum_{j=1}^m \int^jd_j+a_F\in \mI_1\backslash \Psi$ then  $a_\der\neq 0$ and so $d_{-n}\neq 0$ where $n=\deg_\der(a_\der)$. Let 
$$\mu (a):=\mu (d_{-n})\;\; {\rm  and}\;\;\nu (a):=\max\{ s(a),\mu (a)\}.$$

 \begin{theorem}\label{C31Mar24}%\marginpar{C31Mar24}  

\begin{enumerate}
\item  $\pCC_{\mI_1}=\{ a\in \mI_1\backslash \Psi\, | \, \cdot a: P_{1,\leq \nu(a)}'\ra P_{1,\leq \nu (a)+\deg_y(a_y)}'$, $p\mapsto pa$  is an injection$\}$.

\item $\CC'_{\mI_1}=\pCC_{\mI_1}^*$ where $*$ is the involution of the algebra $\mI_1$, see (\ref{*invIn}).
\end{enumerate}
\end{theorem}  
 
\begin{proof} 1. (i) $\pCC_{\mI_1}\cap \Psi=\emptyset$: Suppose that  $a\in  \Psi$.  
Then the map $\cdot a: P_{1,\leq s(a)+1}'\ra P_{1,\leq s(a)}'$, $p\mapsto pa$ is a well-defined map. Since 
$$\dim_K(P_{1,\leq s(a)+1}')=s(a)+2>s(a)+1=\dim_K( P_{1,\leq s(a)}'),$$ $\ker(\cdot a)\neq 0$, $a\not\in \pCC_{\mI_1}$.  Therefore, $\pCC_{\mI_1}\cap \Psi=\emptyset$.

(ii) {\em For each  element} $ a\in \mI_1\backslash \Psi$, $\ker_{P_1'}(\cdot a)\subseteq  P_{1,\leq \nu(a)}'$: Since $a\in \mI_1\backslash \Psi$, $a_y:=\sum_{i=0}^ld_{-i}\der^i\neq 0$ where $d_{-i}\in K[H]$. Suppose that $d_{-l}\neq 0$. Suppose that $p\in \ker_{P_1'}(\cdot a)\backslash  P_{1,\leq \nu(a)}'$, i.e. $\deg_y(p)>\nu(a)$. Then $$\deg(pa)=l+\deg_y(p),$$ a contradiction (since $pa=0$).

Now, statement 1 follows from statements (i) and (ii).

2. Statement 2 follows from statement 1.
\end{proof}

Lemma \ref{b12Mar24} provides examples of rings $R$ such that ${}'\CC_R^{lee}=\emptyset$, ${\CC'}_R^{ree}=\emptyset$ and $\CC_R^{ee}=\emptyset$.

\begin{lemma}\label{b12Mar24}%\marginpar{b12Mar24}
Let $R$ be either $\mI_n$ or $\mS_n$.  Then:  
\begin{enumerate}

\item $\pCCR \cap F_n=\emptyset$, $\CC_R' \cap F_n=\emptyset$ and $\CC_R \cap F=\emptyset$.

\item ${}'\CC_R^{lee}=\emptyset$, ${\CC'}_R^{ree}=\emptyset$ and $\CC_R^{ee}=\emptyset$.

\end{enumerate}
\end{lemma}

\begin{proof} 1. Clearly, every element of the ideal $F$ is a left and right zero-divisor, and statement 1 follows.

2. The ideal $F_n$ is a left and right essential ideal of $R$, and so statement 2 follows from statement 1.
\end{proof}

%%%%%%%%%%%%%%%%%% SECTION 5 %%%%%%%%%%%%%%%%%%%%%

\section{The rings $\pQ(\mA_n)$ and $Q'(\mA_n)$} \label{RI-PQAN} %\marginpar{RI-PQAN}

The aim of the section is to prove Theorem \ref{An18Mar24} and Theorem \ref{20Mar24}.   Theorem \ref{25Mar24} a criterion for  $\pQ (\mA_n)\simeq Q(A_n)$. As a corollary we obtain that $\pQ (\mA_1)\simeq Q(A_1)$ (Theorem \ref{20Mar24}). Theorem \ref{B31Mar24} and Theorem \ref{30Mar15} describe the set $\pCC_{\mA_1}$.
 At the beginning of the section, we recall necessary facts about the Jacobian algebras $\mA_n$ that are used in the proofs. The details can found in \cite{jacalg}. \\

The Weyl algebra $A_n= A_n(K)$ is a simple, Noetherian domain of
Gelfand-Kirillov dimension $\GK (A_n) =2n$. The Jacobian algebra
$\mA_n$ is neither  left nor right Noetherian, it  contains
infinite direct sums of nonzero left and right ideals. This means
that adding the inverses of the commuting regular elements $H_1, \ldots , H_n$ to the Weyl algebra $A_n$ is neither a left nor right Ore localization of the algebra Weyl $A_n$. This fact is a prime reason why the properties of the Jacobian algebras are almost opposite to the ones of the Weyl algebras. 

The algebra $\mA_n$ is a central,
prime algebra of Gelfand-Kirillov dimension $3n$  (\cite[Corollary 2.7]{jacalg}). The canonical involution $\th$ of the Weyl algebra $A_n$ can be uniquely  extended
to the algebra $\mA_n$ (see (\ref{thinv})). So, the
algebra $\mA_n$ is self-dual ($\mA_n\simeq \mA_n^{op}$) and
 its left and right algebraic  properties are the same.  Note
that the {\em Fourier transform} on the Weyl algebra $A_n$ cannot
be lifted to $\mA_n$. Many properties of the algebra
$\mA_n=\mA_1^{\t n }$ are determined by properties of $\mA_1$.
When $n=1$ we usually drop the subscript `1' in $x_1$, $\der_1$,
$H_1$, etc. The algebra $\mA_1$ contains the only proper ideal
$F=\oplus_{i,j\in \N} KE_{ij}$ where
$$
E_{ij}:=\begin{cases}
x^{i-j}(x^j\frac{1}{\der^jx^j}\der^j-x^{j+1}\frac{1}{\der^{j+1}x^{j+1}}\der^{j+1})& \text{if $i\geq j$},\\
(\frac{1}{\der x}\der )^{j-i}(x^j\frac{1}{\der^jx^j}\der^j-x^{j+1}\frac{1}{\der^{j+1}x^{j+1}}\der^{j+1})& \text{if $i<j$}.\\
\end{cases}
$$
As a ring without 1, the ring $F$ is canonically isomorphic to the
ring $M_\infty (K):=\varinjlim M_d(K)= \oplus_{i,j\in \N} KE_{ij}$
of infinite-dimensional matrices where $E_{ij}$ are the matrix
units ($F\ra M_\infty  (K)$, $E_{ij}\mapsto E_{ij}$). This is a
very important fact as we can apply  concepts of
finite-dimensional linear algebra (like trace, determinant, etc)
to integro-differential operators which is not obvious from the
outset. This fact is crucial in finding an inversion formula for
 elements of $\mA_1^*$.

Notice that  $\mA_n = \otimes_{i=1}^n\mA_1(i)\simeq \mA_1^{\t n}$ where $\mA_1(i):=K\langle x_i, \der_i, H_i^{\pm 1}\rangle$ and $H_i=\der_ix_i$. The algebra $\mA_n = \oplus_{\alpha \in \Z^n} \mA_{n, \alpha}$ is
a $\Z^n$-graded algebra where $\mA_{n,\alpha}:=\t_{k=1}^n
\mA_{1,\alpha_k}(k)$ for $\alpha=(\alpha_1, \ldots , \alpha_n)$. For $n=1$, (\cite[Theorem 2.3]{jacalg}), 
%\marginpar{A1i}
\begin{equation}\label{A1i}
\mA_{1,i}=\begin{cases}
x^i\mD_1& \text{if $i\geq 1$},\\
\mD_1& \text{if $i=0$},\\
\mD_1 \der^{-i}& \text{if $i\leq -1$},\\
\end{cases}
\end{equation}
where $\mD_1:= L\oplus (\oplus_{i,j\geq 1}Kx^iH^{-j}\der^i)$ is a
{\em commutative, non-Noetherian} algebra and   
$$L:= K[H^{\pm 1},
(H+1)^{-1}, (H+2)^{-1}, \ldots ]\;\; {\rm  and}\;\;H=\der x=x\der+1.$$
This gives a `compact'
$K$-basis for the algebra $\mA_1$ (and $\mA_n$). This basis
`behaves badly' under multiplication. A more conceptual
(`multiplicatively friendly') basis is given in \cite[Theorem 2.5]{jacalg}, see also \cite[Corollary 2.4]{jacalg} below.

\begin{itemize}

\item (\cite[Corollary 2.4]{jacalg})  $\mD_1 = 
%\mL_1\oplus F_0=
 L\oplus (\oplus_{i\geq 1, j\geq
0}K\rho_{ji })$ {\em where} 
$\rho_{ji}:=x^i\frac{1}{H^j\der^ix^i}\der^i$.

\item For all $i\geq 1$ and $j\geq 0$, $\der^i\rho_{ji }=\frac{1}{H^j}\der^i$ (a direct computation). 

\item (\cite[Corollary 2.7.(10)]{jacalg}) $P_n=K[x_1, \ldots , x_n]$ {\em is the only
faithful, simple $\mA_n$-module.}

\item (\cite[Corollary 3.5]{jacalg}) $ \gp_1:=F\t\mA_{n-1}, \gp_2:= \mA_1\t F\t \mA_{n-2} , \ldots ,
\gp_n :=\mA_{n-1}\t F,$  {\em are precisely the prime ideals  of
height 1 of $\mA_n$.} 

%\item (Corollary \ref{c27Ma7}) {\em Let $\Sub_n$ be the set of all subsets of $\{
%1, \ldots , n\}$. The map $\Sub_n\ra \Spec
%(\mA_n)$, $ I\mapsto \gp_I:= \sum_{i\in I}\gp_i$, $\emptyset
%\mapsto 0$, is a bijection, i.e. any nonzero prime ideal of
%$\mA_n$ is a unique sum of primes of height 1; $|\Spec
%(\mA_n)|=2^n$; the height of $\gp_I$ is $| I|$;  and}
% \item (Lemma \ref{p28Ma7}) $\gp_I\subseteq \gp_J$ {\em iff} $I\subseteq
% J$.
\item (\cite[Corollary 3.15]{jacalg}) $\ga_n:= \gp_1+\cdots +\gp_n$ {\em
is the only prime ideal of $\mA_n$ which is  completely prime;
$\ga_n$ is the only ideal $\ga$ of $\mA_n$ such that $\ga \neq
\mA_n$ and $\mA_n/\ga$ is a Noetherian (resp. left Noetherian,
resp. right Noetherian) ring.}

\item (\cite[Theorem 3.1.(2)]{jacalg}) {\em Each ideal $I$ of $\mA_n$ is an idempotent ideal ($I^2=I$).}

\item (\cite[Corollary 2.7.(4)]{jacalg}) {\em The  ideal $F_n:=F^{\t n}$ is the smallest nonzero ideal of the algebra $\mA_n$.}

\item (\cite[Corollary 2.7.(8)]{jacalg}) {\em  ${}_{\mA_n}F^{\t n }
\simeq P_n^{(\N^n)}$ is a faithful, semi-simple, left
$\mA_n$-module; $F^{\t n }_{\mA_n} \simeq {P_n^{(\N^n)}}_{\mA_n}$
is a faithful, semi-simple, right $\mA_n$-module; ${}_{\mA_n}F^{\t
n }_{\mA_n}$ is a faithful, simple
$\mA_n$-bimodule. }
\end{itemize}

Recall that  $\CP_n$ is a polynomial algebra  $K[H_1, \ldots , H_n]$ in $n$
indeterminates  and  $\sigma=(\sigma_1,...,\sigma_n)$ is an
$n$-tuple  of  commuting automorphisms of $\CP_n$  where
 $\s_i(H_i)=H_i-1$ and $\s_i(H_j)=H_j$, for $i\neq j$.  By \cite[Theorem 2.2]{intdifaut}, $\ga_n$  is the only maximal ideal of
 the Jacobian algebra $\mA_n$. The factor algebra $\CA_n := \mA_n
 / \ga_n$ is the skew Laurent polynomial algebra
 \begin{eqnarray*}
\CA_n &:=&\CL_n [\der_1^{\pm 1}, \ldots , \der_n^{\pm 1}; \tau_1, \ldots
 , \tau_n] = \CL_n [x_1^{\pm 1}, \ldots , x_n^{\pm 1}; \s_1, \ldots
 , \s_n],\\
 \CL_n &:=&K[H_1^{\pm 1}, (H_1\pm 1)^{-1}, (H_1\pm 2)^{-1},
\ldots , H_n^{\pm 1}, (H_n\pm 1)^{-1}, (H_n\pm 2)^{-1}, \ldots ],
\end{eqnarray*}
where $\tau_i(H_j)= H_j+\d_{ij}$,  $\d_{ij}$ is the Kronecker delta, $z_i=\der_i^{-1}$ 
and  $\s_i=\tau_i^{-1}$. The algebra $B_n$ is a subalgebra of
$\CA_n$. 
Let  $S_n$ be a multiplicative submonoid
of $\CP_n$ generated by the elements $H_i+j$, $i=1, \ldots , n$,
and $j\in \mathbb{Z}$. Then $S_n$ is an Ore set for the Weyl algebra  $A_n$, the algebras $B_n$ and  the polynomial algebra $\CP_n$ such that  $\CA_n:=S_n^{-1}A_n\simeq S_n^{-1}B_n$  and 
$$ S_n^{-1}\CP_n=K[H_1^{\pm 1}, (H_1 \pm 1)^{-1}, (H_1 \pm 2)^{-1}, \ldots ,
H_n^{\pm 1}, (H_n \pm 1)^{-1}, (H_n \pm 2)^{-1}, \ldots ].$$ We identify the Weyl algebra
$A_n$ with a subalgebra of $\CA_n$ via the monomorphism 
$$A_n\ra \CA_n, \;\; x_i\mapsto x_i,\;\; \der_i\mapsto  H_ix_i^{-1}, \;\;
i=1, \ldots , n.$$ 
The Weyl algebra $A_n$ is a Noetherian domain. So, by  Goldie's Theorem, the (left and right)  quotient ring of $A_n$, $Q(A_n)$, is a division ring.  Then the algebra $\CA_n$ is a $K$-subalgebra of $Q(A_n)$
generated by the elements $x_i$, $x_i^{-1}$, $H_i$ and $H_i^{-1}$,
$i=1, \ldots , n$ since
$$ (H_i\pm j)^{-1}=x_i^{\mp j}H_i^{-1}x_i^{\pm j}, \;\; i=1, \ldots , n\;\; {\rm and}\;\; j\in \N.$$
Clearly, $\CA_n \simeq \CA_1^{\t n}$.\\

{\bf The involution $\th$ on $\mA_n$}. %Let $K$ be a commutative$\Q$-algebra. 
The Weyl algebra $A_n$ admits the {\em involution}
$$\th : A_n\ra A_n, \;\; x_i\mapsto \der_i, \;\; \der_i\mapsto
x_i, \;\; i=1, \ldots , n.$$ 
%i.e. it is a $K$-algebra anti-isomorphism ($\th (ab) = \th (b) \th (a)$) such that  $\th^2=
%{\rm id}_{A_n}$. 
The involution $\th$ is uniquely extended to
an involution of $\mA_n$ by the rule %\marginpar{thinv}
\begin{equation}\label{thinv}
\th : \mA_n\ra \mA_n, \;\; x_i\mapsto \der_i, \;\; \der_i\mapsto
x_i,\;\;  \th (H_i^{-1})= H_i^{-1},  \;\; i=1, \ldots , n.
\end{equation}
Uniqueness is obvious: $\th (H_i) = \th (\der_ix_i)= \th (x_i )
\th (\der_i) = \der_ix_i= H_i$ and so $\th (H_i^{-1}) = H_i^{-1}$.  So, the algebra
$\mA_n$ is self-dual and   left
and right algebraic  properties of the algebra $\mA_n$ are the same.

The polynomial algebra $P_n$ is a left $A_n$-module
 where $\der_i *f:=\frac{\der f}{\der x_i}$ for all $i=1, \ldots , n$ and $f\in P_n$. The left $A_n$-module $P_n$ is isomorphic to the $A_n$-module $$A_n/A_n(\der_1, \ldots , \der_n)\simeq P_n\b1\simeq {}_{P_n}P_n \;\; {\rm where} \;\; \b1:=1+A_n(\der_1, \ldots , \der_n).$$
The maps $H_i: P_n\ra P_n$ are invertible for all $i=1, \ldots , n$ since $H_i x^\alpha \b1 =(\alpha_i+1)x^\alpha$ for all $\alpha =(\alpha_1, \ldots , \alpha_n)\in \N^n$ where $x^\alpha:=\prod_{i=1}^nx_i^{\alpha_i}$. Therefore, the polynomial algebra $P_n$ is also a left $\mA_n$-module which is isomoprphic to 
$$\mA_n/\mA_n(H_1-1, \ldots , H_n-1,\der_1, \ldots , \der_n)\simeq P_n\b1, \;\; {\rm where}\;\;\b1:=1+\mA_n(H_1-1, \ldots , H_n-1,\der_1, \ldots , \der_n),$$ by \cite[Theorem 2.3]{jacalg}. 
Using the involution
$\th$ on $\mA_n$,  the polynomial algebra $P_n$
can be seen as the {\em right} $\mA_n$-module by the rule 
$$pa:= \th (a) p\;\; {\rm  for\;
all}\;\;p\in P_n\;\; {\rm  and}\;\;a\in \mA_n.$$ By \cite[Corollary 2.7.(10)]{jacalg}),  $P_n=K[x_1, \ldots , x_n]$  {\em is the only
faithful, simple, right $\mA_n$-module}. Let $P_n':=\th (P_n)=K[\der_1, \ldots, \der_n]$, a polynomial algebra in $n$ variables. Clearly, $$(P_n)_{\mA_n}=\th (P_n\b1)=\widetilde{1}\th (P_n)=\widetilde{1}P_n'\simeq \mA_n/(H_1-1, \ldots , H_n-1,x_1, \ldots, x_n)\mA_n\simeq (P_n')_{P_n'}$$
where $\widetilde{1}:=1+(H_1-1, \ldots , H_n-1,\der_1, \ldots, \der_n)\mA_n$.

For $n=1$, the set $F$ is the only proper ideal of $\mA_1$, hence $\th (F)
= F$. Moreover, %\marginpar{thEij}
\begin{equation}\label{thEij}
\th (E_{ij})=\frac{i!}{j!}\, E_{ji}
\end{equation}
where $0!:=1$. The ring $F=\oplus_{i,j\in \N} KE_{ij}$
is equal to the matrix ring $M_\infty (K):= \cup_{d\geq 1} M_d(K)$
where $M_d(K):=\oplus_{0\leq i,j\leq d-1}KE_{ij}$. The ring
$F=M_\infty (K)$ admits the canonical involution which is the
 transposition $(\cdot )^t: E_{ij}\mapsto E_{ji}$. Let $D_!$ be the infinite
 diagonal matrix ${\rm diag} (0!, 1!, 2!, \ldots )$. Then, for $u\in
 F=M_\infty (K)$,
%\marginpar{1thEij}
\begin{equation}\label{1thEij}
\th (u)=D_!^{-1} u^tD_!.
\end{equation}
Note that $D_!\not\in M_\infty (K)$. For  $n\geq 1$, $F_n:=F^{\t n } = \oplus_{\alpha , \beta \in
\N^n} K E_{\alpha \beta}=M_\infty (K)^{\t n }$ where $E_{\alpha
\beta}:= \t_{i=1}^n E_{\alpha_i\beta_i}$. By (\ref{thEij}),
%\marginpar{nthEij}
\begin{equation}\label{nthEij}
\th (E_{\alpha \beta})= \frac{\alpha !}{\beta !}\, E_{\beta
\alpha},
\end{equation}
%\marginpar{thFn}
\begin{equation}\label{thFn}
\th (F^{\t n } ) =  F^{\t n }.
\end{equation}
Let $D_{n,!}:= D_!^{\t n }$. Then, for $u\in F^{\t n }$,
%\marginpar{2thFn}
\begin{equation}\label{2thFn}
\th (u) =D_{n,!}^{-1}u^tD_{n,!}
\end{equation}
where $(\cdot )^t:M_\infty (K)^{\t n } \ra M_\infty (K)^{\t n }$,
$E_{\alpha \beta}\mapsto E_{\beta \alpha}$, is the transposition
map.

Consider the bilinear, symmetric, nondegenerate form $(\cdot,
\cdot ): P_n\times P_n\ra K$ given by the rule $(x^\alpha ,
x^\beta ) := \alpha ! \d_{\alpha, \beta }$ for all $\alpha , \beta
\in \N^n$. Then, for all $p,q\in P_n$ and $a\in \mA_n$,
%\marginpar{paq=t}
\begin{equation}\label{paq=t}
(p,aq)= (\th (a) p , q).
\end{equation}

The Weyl algebra $A_n$ admits, the so-called,  {\em Fourier
transform},  which is the $K$-algebra automorphism $$\CF : A_n\ra A_n, \;\; x_i\mapsto \der_i, \;\;\der_i\mapsto - x_i \;\; {\rm for} \;\;i=1, \ldots , n.$$
Since $\CF (H_i) = -(H_i-1)$, $H_i$ is a unit of $\mA_n$ and
$H_i-1$ is not,  one {\em cannot} extend the Fourier transform to
$\mA_n$.

For all $i=1, \ldots , n$,  $\der_i\in \pCC_{\mA_n}$. Recall that if $n=1$ then $\der^i\rho_{ji }=\frac{1}{H^j}\der^i$  for all $i\geq 1$ and $j\geq 0$.   It follows that 
$$S_\der:=S_{\der_1, \ldots , \der_n}:=\{
\der_1^{\alpha_1}\cdots \der_n^{\alpha_n}\, | \, (\alpha_i)\in
\N^n\} \in \Den_l(\mA_n, \ga_n), \;\; S_\der\subseteq \pCC_{\mA_n}  \;\; {\rm and}\;\; S_\der^{-1}\mA_n \simeq \CA_n .$$

{\bf The ring $\pQ (\mA_n)$.} The ring 
 \begin{eqnarray*}
\D_n&:=&\CL_n^+[\der_1, \ldots , \der_n; \tau_1, \ldots , \tau_n],\;\; {\rm where}\\
 \CL_n^+ &:=&K[H_1^{\pm 1}, (H_1+ 1)^{-1}, (H_1+2)^{-1},
\ldots , H_n^{\pm 1}, (H_n+1)^{-1}, (H_n+2)^{-1}, \ldots ]\simeq L^{\t n},
\end{eqnarray*}
 is a skew polynomial ring where $\tau_i(H_j)=H_i+\d_{ij}$,  $\D_n^0:=\D_n\backslash \{ 0\}$, ${}'\D_n^0:=\D_n\cap \pCC_{\mA_n}$, 
$${}'\D_n:=\D_n\cap \pS (\mA_n)\;\; {\rm  and }\;\;\widetilde{{}'\D_n}:=\{c\in \D_n\, | \, \der^\alpha c\in \pS (\mA_n)\; {\rm  for\; some}\;\;\alpha \in \N^n\}.
$$
 Notice that  ${}'\D_n\subseteq 
{}'\D_n^0\subseteq \D_n^0$. Theorem \ref{An18Mar24} produces explicit left denominators sets   $S\in \Den_l(\mA_n, \ga_n)$ such that  $S^{-1}\mA_n\simeq \pQ (\mA_n)$. 
  By Theorem \ref{An18Mar24}, there are inclusions in the set $\Den_l(\mA_n, \ga_n)$ apart from $S_\der$: 
\begin{equation}
S_\der\subseteq {}'\D_n  \subseteq \pS (\mA_n)\subseteq \pS (\mA_n)+\ga_n\;\; {\rm and}\;\;{}'\D_n\subseteq \widetilde{{}'\D_n}\subseteq '\D_n+\ga_n \subseteq \pS (\mA_n)+\ga_n.
\end{equation}

\begin{theorem}\label{An18Mar24}%\marginpar{An18Mar24}
 
\begin{enumerate}

\item ${}'\D_n\in \Den_l(\mA_n, \ga_n)$ and ${}'\D_n^{-1}\mA_n\simeq \pQ (\mA_n)$. Furthermore, the  subset  ${}'\D_n$ of $\mA_n$  is a left denominator set of $\mA_n$ which is the  largest  left denominator set that is contained in the multiplicative set $\pCC_{\mA_n}\cap \D_n$.

\item $  {}'\D_n+\ga_n, \pS (\mA_n)+\ga_n \in \Den_l(\mA_n, \ga_n)$ and 
$$ 
 \Big({}'\D_n+\ga_n\Big)^{-1}\mA_n\simeq \Big( \pS (\mA_n)+\ga_n\Big)^{-1}\mA_n\simeq \pQ (\mA_n).  
$$

\item ${}'\D_n=\overline{{}'\D_n},\overline{{}'\D_n+\ga_n}, \overline{\pS (\mA_n)+\ga_n} \in \Den_l(\CL_n, 0)$ and 
$$ 
{}'\D_n^{-1}\CL_n\simeq \overline{{}'\D_n+\ga_n}^{-1}\CL_n\simeq \overline{\pS (\mA_n)+\ga_n}^{-1}\CL_n\simeq  \pQ (\mA_n)  
$$
where $\bS:= \pi_{\ga_n}(S)$ and  $\pi_{\ga_n}: \mA_n\ra \CL_n=\mA_n/\ga_n$, $r\mapsto \br := r+\ga_n$.

\item $\widetilde{{}'\D_n}\in \Den_l(\mA_n, \ga_n)$ and $\widetilde{{}'\D_n}^{-1}\mA_n\simeq \pQ (\mA_n)$. $\overline{\widetilde{{}'\D_n}}\in \Den_l(\CL_n, 0)$ and $\overline{\widetilde{{}'\D_n}}^{-1}\CL_n\simeq \pQ (\mA_n)$.

\item $\widetilde{{}'\D_n}+\ga_n\in \Den_l(\mA_n, \ga_n)$ and $\Big(\widetilde{{}'\D_n}+\ga_n\Big)^{-1}\mA_n\simeq \pQ (\mA_n)$.

\item ${}'\D_n \in \Den_l(\D_n, 0)$ and 
$
{}'\D_n^{-1}\D_n\simeq  \pQ (\mA_n)$.

\item ${}''\D_n:=\pS (\mA_n)\cap (\pCC_{\D_n+\ga_n}+\ga_n)+\ga_n\in \Den_l(\mA_n, \ga_n)$, $
{}'\D_n\subseteq 
{}''\D_n$ and $
{}''\D_n^{-1}\mA_n\simeq  \pQ (\mA_n)$.
\end{enumerate}
\end{theorem}

\begin{proof} 1. (i)  ${}'\D_n\in \Den_l(\mA_n, \ga_n)$ {\em and} ${}'\D_n^{-1}\mA_n\simeq\pQ (\mA_n)$: By the definition, the set    ${}'\D_n$ is a multiplicative subset of $\pS (\mA_n)\subseteq \mA_n$. It follows from the inclusions $S_\der \subseteq {}'\D_n\subseteq \pS (\mA_n)$ that 
$$\ga_n = \ass_l(S_\der )\subseteq \ass_l({}'\D_n)\subseteq \ass_l(\pS (\mA_n))=\ga_n, $$
and so $\ass_l({}'\D_n)=\ga_n=\ass_l(\pS (\mA_n))$. For each element $s\in \pS (\mA_n)$, there is an element $\alpha\in \N^n$ such that $\der^\alpha s\in {}'\D_n$. Notice that $\der^\alpha \in S_\der \subseteq {}'\D_n$. 
Now, the statement (i) follows from Lemma \ref{c16Mar15} where $T=\pS (\mA_n)\in \Den_l(\mA_n,\ga_n)$ and  $S={}'\D_n$.

(ii) {\em The  subset  ${}'\D_n$ of $\mA_n$  is a left denominator set of $\mA_n$ which is the  largest  left denominator set that is contained in the multiplicative set $\pCC_{\mA_n}\cap \D_n$}: Let $T$ be a left denominator set of $\mA_n$ which is the  largest  left denominator set that is contained in the multiplicative set $\pCC_{\mA_n}\cap \D_n$. By the statement (i),  ${}'\D_n\in \Den_l(\mA_n, \ga_n)$. Clearly,  ${}'\D_n\subseteq  \pCC_{\mA_n}\cap \D_n$ and so ${}'\D_n\subseteq T$. Since $T\subseteq \pCC_{\mA_n}$ and $\pS (\mA_n)$ is a the largest left denominator set in $\pCC_{\mA_n}$, we have the inclusion $T\subseteq \D_n\cap\pS(\mA_n)={}'\D_n$. Therefore, $T={}'\D_n$.

2. By statement 1, ${}'\D_n, \pS (\mA_n)\in\Den_l(\mA_n, \ga_n)$ and ${}'\D_n^{-1}\mA_n\simeq  \pS (\mA_n)^{-1}\mA_n=\pQ (\mA_n)$. Now, statement 2 follows from Corollary \ref{a19Mar24}.(2).   

3.  Statement 3 follows at once from statement 2 (If $S\in \Den_l(R,\ga)$ then $\bS :=S+\ga \in \Den_l(\bR, 0)$ and $S^{-1}R\simeq \bS^{-1}\bR$ where $\bR := R/\ga$).

4. By the definition, the set $\widetilde{{}'\D_n}$ is a multiplicative set such that ${}'\D_n\subseteq \widetilde{{}'\D_n}\subseteq \pS (\mA_n)$. Therefore,  $\ga_n=\ass_l({}'\D_n)\subseteq \ass_l(\widetilde{{}'\D_n})\subseteq \ass_l(\pS (\mA_n))=\ga_n$, and so $\ass_l(\widetilde{{}'\D_n})=\ga_n$. Now, the first part of statement 4 follows from the definition of the set $\widetilde{{}'\D_n}$ and Lemma \ref{c16Mar15} where $S=\widetilde{{}'\D_n}$ and $T=\pS(\mA_n)$.

The second  part of statement 4 follows from the first one.

5. Statement 5 follows from statement 4.

6. By statement 3,  ${}'\D_n \in \Den_l(\CL_n, 0)$ and $ {}'\D_n^{-1}\CL_n\simeq  \pQ (\mA_n)$.  Now, statement 6 follows from  Lemma \ref{b18Mar24}.(2) where $R=\D_n$,  $R'=\CL_n$ and  $S={}'\D_n$ (The $R$-module $R'/R$ is $S_\der$-torsion. Hence it is also ${}'\D_n$-torsion as $S_\der\subseteq {}'\D_n$).

7. By the definition, the subset ${}''\D_n$ of $\mA_n$ is a multiplicative set such that ${}'\D_n\subseteq {}''\D_n\subseteq \pS(\mA_n)$, see statements 1 and 2. Hence, 
$\ga_n=\ass_l({}'\D_n)\subseteq \ass_l({}''\D_n)\subseteq \ass_l(\pS(\mA_n))=\ga_n$, and so $\ass_l({}''\D_n)=\ga_n=\ass_l(\pS(\mA_n))$. Notice that $S_\der\subseteq {}'\D_n\subseteq {}''\D_n$ and for each element $s\in {}''\D_n$ there is an element $\alpha\in \N^n$ such that $\der^\alpha s\in {}'\D_n\subseteq {}''\D_n$. Now, statement 7 follows from 
 Lemma \ref{c16Mar15}  where $S={}''\D_n$ and  $T=\pS(\mA_n)$.
\end{proof}

In order to prove Theorem \ref{25Mar24},  we need the following two lemmas that are also interesting on their own. 

\begin{lemma}\label{a29Mar24}%\marginpar{a29Mar24}

\begin{enumerate}

\item Every nonzero ideal of the algebra $\mA_n$ has zero left and right annihilator.

\item Every nonzero ideal of the algebra $\mA_n$ is an essential left and right ideal of $\mA_n$.

\end{enumerate}
\end{lemma}

\begin{proof}  The ideal $F_n$ is the smallest nonzero ideal of the algebra $\mA_n$. So, it suffices to prove statements 1 and 2 for the ideal $F_n$. 
 Suppose that the left or right annihilator of $F_n$ is   a nonzero ideal of $\mA_n$.  Hence, it contains the idempotent ideal $F_n$, and so their product, which is the zero ideal,  contains $F_n^2=F_n\neq 0$, a contradiction. 

Let $I$ and $J$ be  left and right ideals of $\mA_n$, respectively. By statement 1, $I\supseteq F_nI\neq 0$ and $J\supseteq JF_n\neq 0$, and statement 2 follows. 
\end{proof}

Lemma \ref{b29Mar24} gives a characterization of left regular elements of the ring $\mA_n$. By applying the involution $\th$, we obtain a similar characterization of right regular elements of the ring $\mA_n$ ($a\in \pCC_{\mA_n}$ iff $\th (a)\in \CC_{\mA_n}'$).

\begin{lemma}\label{b29Mar24}%\marginpar{b29Mar24}
Let $a\in \mA_n$. Then the following statements are equivalent:
\begin{enumerate}

\item $a\in \pCC_{\mA_n}$.

\item $a\in {}'\CC_{F_n}$.

\item  $a\in {}'\CC_{P_n'}$ where $P_n'=K[\der_1, \ldots,  \der_n]$ is the only simple, faithful, right $\mA_n$-module.

\item $\th (a)\in \CC_{P_n}'$  where $P_n=K[x_1, \ldots  ,x_n]$ is the only simple, faithful, left $\mA_n$-module.

\end{enumerate}
\end{lemma}

\begin{proof} $(1\Leftrightarrow 2)$ The equivalence follows from Corollary \ref{d15Mar24}.(1) and   Lemma \ref{a29Mar24}.(2).

 $(2\Leftrightarrow 3)$ The equivalence follows from Corollary \ref{e15Mar24}.(1) and the fact that $(F_n)_{\mA_n}\simeq (P_n')^{(\N^n)}$  (\cite[Corollary 2.7.(8)]{jacalg}).

 $(3\Leftrightarrow 4)$ The equivalence follows from the fact that $P_n'=\th (P_n)$.
 \end{proof}

{\bf Criterion for  $\pQ (\mA_n)\simeq Q(A_n)$.}
 The algebras $\D_n$, $B_n$ and $\CA_n$ are Noetherian domains. By Goldie's Theorem, their quotient rings are division rings. 
It follows from the inclusions $A_n\subseteq B_n\subseteq \CA_n\subseteq Q(A_n)$ and  $\D_n\subseteq S_\der^{-1}\D_n\simeq  \CA_n\subseteq Q(A_n)$ that 
%\marginpar{AQns-1}
\begin{equation}\label{AQns-1}
Q(\D_n)= Q(B_n)=Q(\CA_n)= Q(A_n).
\end{equation}

\begin{theorem}\label{25Mar24}%\marginpar{25Mar24}
The following statements are equivalent:
\begin{enumerate}

\item $\pQ (\mA_n)\simeq Q(A_n)$.

\item $\pQ_{l,cl} (\mA_n)\simeq Q(A_n)$.

\item The set $\overline{\pS (\mA_n)}$ is left  dense in  $\CA_n\backslash \{0\}$.

\item The set $\overline{\pCC_{\mA_n}}$ is left  dense in  $\CA_n\backslash \{0\}$.

\item The set ${}'\D_n$ is left  dense in $\CA_n\backslash \{0\}$.

\item The set ${}'\D_n$ is left  dense in $\D_n^0$.

\item For each element $s\in \D_n^0$, there is an element $s'\in \D_n^0$ such that $s's\in  \pCC_{F_n}$.

\item For each element $s\in \D_n^0$, there is an element $s'\in \D_n^0$ such that $s's\in \pCC_{P_n'}$ where $P_n'=K[\der_1, \ldots , \der_n]$ is the unique simple faithful right $\mA_n$-module.

\end{enumerate}
\end{theorem}

\begin{proof} $(1 \Leftrightarrow 2 \Leftrightarrow 3 \Leftrightarrow 4)$ The equivalence  $(1 \Leftrightarrow 2)$ follows from Theorem \ref{8Mar15}.(2). The equivalences  $(1 \Leftrightarrow 3)$  and  $(2\Leftrightarrow 4)$ follows from Theorem \ref{28Feb15} and Theorem \ref{8Mar15}.(2) (since $\ass_l(\pS (\mA_n))=\ga_n$ is a prime ideal of $\mA_n$, the algebra $\mA_n/\ga_n=\CA_n$ is a domain  and $Q(\mA_n/\ga_n)=Q(\CA_n)$ is a division ring).

$(1 \Leftrightarrow 5)$ By Theorem \ref{An18Mar24}.(1),  ${}'\D_n\in  \Den_l(\mA_n, \ga_n)$ and ${}'\D_n^{-1}\mA_n\simeq \pQ (\mA_n)$.
 Since the set ${}'\D_n$ is dense in  $\pS (\mA_n)$, the equivalence $(1 \Leftrightarrow 5)$ holds iff the equivalence $(1 \Leftrightarrow 3)$ holds.
 
$(5 \Leftrightarrow 6)$ Recall that ${}'\D_n, \D_n^0\in\Den_l (\D_n,0)$, ${}'\D_n\subseteq \D_n^0$  and  $Q(\D_n)=Q(A_n)$, see (\ref{AQns-1}). By \cite[Lemma 3.5.(3)]{Clas-lreg-quot}, ${}'\D_n^{-1}\D_n\simeq (\D_n^0)^{-1}\D_n=Q(A_n)$ iff
 the set ${}'\D_n$ is left  dense in $\D_n^0$.

$(6 \Leftrightarrow 7)$ By  Lemma \ref{b29Mar24}, the inclusion  $s's\in  \pCC_{F_n}$ is equivalent to the inclusion $s's\in  \pCC_{\mA_n}$. Hence, statement 7 is equivalent to the statement that 
  for each element $s\in \D_n^0$, there is an element $s'\in \D_n^0$ such that $s's\in   \pCC_{\mA_n}$, i.e. $s's\in  \D_n\cap  \pCC_{\mA_n}={}'\D_n$, i.e. it is equivalent to statement 6.

$(7\Leftrightarrow 8)$ The equivalence follows from  
 Lemma \ref{b29Mar24}.
\end{proof}

{\bf The ring $\pQ (\mA_1)$.} As an application of Theorem \ref{25Mar24} we obtain Theorem \ref{20Mar24}.

\begin{theorem}\label{20Mar24}%\marginpar{20Mar24}
 $\pQ (\mA_1)\simeq Q(A_1)$. 
\end{theorem}

\begin{proof}  (i) {\em A nonzero rational function $\phi \in \CL_1$ belongs to the set $\pCC_{\mA_1}$ iff it has no root in the set $\N_+=\{ 1,2, \ldots\}$}: The right $\mA_1$-module $F$ is an essential right ideal of the algebra $\mA_1$ (\cite[Corollary 2.7.(6)]{jacalg}). Therefore, $\phi \in \pCC_{\mA_1}$ iff the map $\cdot \phi_F: F\ra F$, $f\mapsto f\phi$ is an injection. 
The right $\mA_1$-module $F=\oplus_{i\in \N}E_{i0}\mA_1\simeq \Big(E_{00}\mA_1 \Big)^{(\N)}$ is a direct sum of countably many copies of the right $\mA_1$-module $$E_{00}\mA_1=E_{00}F\simeq E_{00}K[\der]\simeq K[\der]_{K[\der]}\;\; {\rm  and}\;\;E_{00}\der^i H=E_{00}\der^i (i+1)\;\; {\rm for\; all}\;\;i\geq 0,$$
and the statement (i) follows.

(ii) {\em For each nonzero rational function $\phi \in \CL_1$, $\tau^i(\phi)\in \pCC_{\mA_1}$ for all $i\gg 1$ where $\tau (H)=H+1$}: The statement (ii) follows from the statement (i).

(iii) {\em For each nonzero element $d\in \D_n$, $\der^i d\in {}'\D_n$  for all $i\gg 1$}: The element $d$ is a unique sum $\phi_n\der^n+\phi_{n+1}\der^{n+1}+\cdots +\phi_m\der^m$ where $\phi_i \in \CL_1$ and $\phi_n\neq 0$. By the statement (ii), $\tau^i(\phi_n)\in \pCC_{\mA_1}$ for some $i\geq 0$. Therefore, the map $\cdot \tau^i(\phi_n): F\ra F$, $f\mapsto f\tau^i(\phi_n)$ is an injections. Hence, the map 
$\cdot \der^id: F\ra F$, $f\mapsto f\der^id$ is also an injections since 
$$\der^id=\tau (\phi_n)^i\der^{n+i}+\cdots +\tau (\phi_m)^i\der^{m+i}\;\; {\rm and}\;\;\tau (\phi_n)^i\der^{n+i}\in \pCC_{\mA_1}.$$
 Therefore, $\der^id\in \pCC_{\mA_1}$.

(iv) {\em The set  ${}'\D_1$ is left  dense in $\CA_1\backslash \{0\}$}: Clearly, the set  ${}'\D_1\backslash \{ 0\}$ is left dense in $\CA_1\backslash \{0\}$. Now, the statement (iv) follows from the statement (iii).

The theorem follows from Theorem \ref{25Mar24} and the fact that the set  ${}'\D_1$ is left dense in $\CA_1\backslash \{0\}$, the statement (iv).
\end{proof}

\begin{corollary}\label{b20Mar24}%\marginpar{b20Mar24}
 $Q' (\mA_1)\simeq Q(A_1)$. 
\end{corollary}

\begin{proof} The result follows from Theorem \ref{20Mar24}: $Q' (\mA_1)=\th(\pQ (\mA_1))\simeq \th(Q(A_1)=Q(\th(A_1))= Q(A_1)$. 
\end{proof}

{\bf Descriptions of the sets $\pCC_{\mA_1}$ and $\CC_{\mA_1}'$.}

(i) $\pCC_{\mA_1}\subseteq \mA_1\backslash F$: It is obvious that every element of the ideal $F=\oplus_{i,j\in \N}KE_{ij}\simeq M_\infty (K)$ is a left and right zero divisor of the algebra $F$ (without 1) and of $\mA_1$.

(ii) {\em For each nonzero element $d\in \mA_1\backslash F$, $\der^i d \in \D_1^0$  for some}  $i\in \N$: The statement (ii) follows (\ref{A1i}) and the equality $\mD_1=L\oplus(\oplus_{i,j\geq 1}Kx^iH^{-j}\der^i)$.

(iii) {\em For each nonzero element $d\in \mA_1\backslash F$, $\der^i d\in \pCC_{\mA_1}$  for some  $i\in \N$}: The statement (iii) follows from the statement (ii) and  the statement (iii) of the proof of Theorem \ref{20Mar24}. 

 Then the well-defined map
%\marginpar{mAn-GNd}
\begin{equation}\label{mAn-GNd}
d: \mA_1\backslash F \ra \N , \;\; a\mapsto d(a) :=\min \{ i\in \N \, | \, \der^ia\in \pCC_{\mA_1}\}
\end{equation}
is called the {\em left regularity degree function} and the natural number $d(a)$ is called the {\em left regularity degree} of $a$. For each element $a\in \mA_1\backslash F$, $d(a)$ can be found in finitely many steps, see the proof of Theorem \ref{20Mar24}.  %where the explicit expression  (\ref{daca})  is given for $d(a)$.
Now, Theorem \ref{30Mar15}.(1) follows. Then Theorem \ref{30Mar15}.(2) follows from Theorem \ref{30Mar15}.(1). 

\begin{theorem}\label{30Mar15}%\marginpar{30Mar15}

\begin{enumerate}
\item $\pCC_{\mA_1} = \{ \der^{d(a)}a\, | \, a\in \mA_1\backslash F\}$.
\item $\CC_{\mA_1}'= \th (\pCC_{\mA_1})$.
\end{enumerate}
\end{theorem}

By (\ref{A1i}), each element $a\in \mA_1$ is  a unique sum $a=\sum_{i=0}^ld_{-i}\der^i+\sum_{j=1}^m x^jd_j+a_F$ where $d_k\in \mD_1$, $k=-l, \ldots , m$  and $a_F\in F$. Let $a_\der:=\sum_{i=0}^ld_{-i}\der^i$. The integer
 $$s(a_F):=
\begin{cases}
\min \{ n\in \N \, | \, a_F\in \bigoplus_{i,j=0}^nKe_{ij}\} & \text{if }a_F\neq 0,\\
-1& \text{if } a_F=0.\\
\end{cases}  
$$
is called the {\em size} of the element $a_F\in F$. The integer $s(a) := s(a_F)$ is called the {\em size} of the element $a$. For each $i\in \N$, let $P_{1,\leq i}':=\{ a\in P_1'\, | \, \deg_y(a)\leq i\}$ where $\deg_y\der$ is the degree of the polynomial $a\in P_1'=K[\der]$ in the variable $\der$. Let  
$$L^\perp:=\bigoplus_{i,j\geq 1}Kx^iH^{-j}\der^i\;\; {\rm  and}\;\;\Xi := \bigg(\bigoplus_{i\geq 0}L^\perp\der^i \oplus \bigoplus_{i\geq 1}x^i\mD_1\bigg)\bigcup \bigg(\bigoplus_{i\geq 1}x^i\mD_1+F\bigg).$$ Then $\mD_1=L\oplus L^\perp$. For  a nonzero element $l^\perp=\sum_{i,j\geq 1}\l_{ij}x^iH^{-j}\der^i\in L^\perp$, let 
$$
\phi (l^\perp):=\sum_{i,j\geq 1}\l_{ij}\frac{(H-i)(H-i+1)\cdots (H-1)}{(H-i)^j}\;\;  \d (l^\perp):=\max \{ i\geq 1\, | \, \l_{ij}\neq 0\;\;{\rm  for\; some}\;\; j\geq 1\}
$$ 
and $\d (0):=0$.  Then for all $k\geq \d (l^\perp)$,
$$\der^kl^\perp=\tau (\phi (l^\perp))\der^k\;\; {\rm and}\;\; \tau (\phi (l^\perp))\in L$$
where $\tau (H)=H+1$.
 For each nonzero element $d=l+l^\perp\in \mD_1=L\oplus D^\perp$, where $l\in \L$ and $l^\perp\in L^\perp$, 
 let 
 $$\mu (d):=\min\{ i\geq \d (l^\perp)\, | \,{\rm the \; rational\; function}\;\; \tau^i(l+\phi (l^\perp))\in L\;\; {\rm has\; no\; root\;in}\;\; \N_+\}.$$ 
If $a=\sum_{i=0}^nd_{-i}\der^i+\sum_{j=1}^m x^jd_j+a_F\in \mA_1\backslash \Xi$ then  $a_\der\neq 0$ and so $d_{-n}\neq 0$ where $n=\deg_\der(a_\der)$. Let 
$$\mu (a):=\mu (d_{-n})\;\; {\rm  and}\;\;\nu (a):=\max\{ s(a),\mu (a)\}.$$

 \begin{theorem}\label{B31Mar24}%\marginpar{B31Mar24}

\begin{enumerate}
\item  $\pCC_{\mA_1}=\{ a\in \mA_1\backslash \Xi\, | \, \cdot a: P_{1,\leq \nu(a)}'\ra P_{1,\leq \nu (a)+\deg_y(a_y)}'$, $p\mapsto pa$  is an injection$\}$.

\item $\CC'_{\mA_1}=\th (\pCC_{\mA_1})$ where $\th$ is the involution of the algebra $\mA_1$, see (\ref{thinv}).
\end{enumerate}
\end{theorem}  
 
\begin{proof} 1. (i) $\pCC_{\mA_1}\cap \Xi=\emptyset$: Let $a\in \Xi$. We have to show that $a\not\in \pCC_{\mA_1}$.
 Suppose that  $a\in \bigoplus_{i\geq 0}L^\perp\der^i \oplus \bigoplus_{i\geq 1}x^i\mD_1$. Then $1\in \ker(\cdot a)$, and so $a\not\in \pCC_{\mA_1}$.
  
Suppose that  $a\in \bigoplus_{i\geq 1}x^i\mD_1+ F$.  
Then the map $\cdot a: P_{1,\leq s(a)+1}'\ra P_{1,\leq s(a)}'$, $p\mapsto pa$ is a well-defined map. Since 
$$\dim_K(P_{1,\leq s(a)+1}')=s(a)+2>s(a)+1=\dim_K( P_{1,\leq s(a)}'),$$ $\ker(\cdot a)\neq 0$, $a\not\in \pCC_{\mA_1}$.  Therefore, $\pCC_{\mA_1}\cap \Xi=\emptyset$.

(ii) {\em For each  element} $ a\in \mA_1\backslash \Xi$, $\ker_{P_1'}(\cdot a)\subseteq  P_{1,\leq \nu(a)}'$: Since $a\in \mA_1\backslash \Xi$, $a_y:=\sum_{i=0}^ld_{-i}\der^i\neq 0$ where $d_{-i}\in \mD_1$. Suppose that $d_{-l}\neq 0$. Suppose that $p\in \ker_{P_1'}(\cdot a)\backslash  P_{1,\leq \nu(a)}'$, i.e. $\deg_y(p)>\nu(a)$. Then $$\deg(pa)=l+\deg_y(p),$$ a contradiction (since $pa=0$).

Now, statement 1 follows from statements (i) and (ii).

2. Statement 2 follows from statement 1.
\end{proof}

%%%%%%%%%%%%%%%%%% SECTION  %%%%%%%%%%%%%%%%%%%%%

%\section{ AAA} \label{AAA} %\marginpar{AAA}

\begin{theorem}\label{27Mar24}%\marginpar{27Mar24}Let $\CP_n=K[H_1, \ldots , H_n]$ be a polynomial algebra over a field $K$ of characteristic zero and  $\s_i\in \Aut_K(\CP_n)$ where $\s_i (H_j)=H_j-\d_{ij}$ for $i,j=1, \ldots , n$ and $\d_{ij}$ is the Kronecker delta.  Let  $\phi\in \CP_n\backslash \{ 0\}$ and $\l_d H^d$ be the  leading term of the polynomial  $\phi$ with respect to  the lexicographic ordering $H_1<\cdots <H_n$ where $\l_d\in K^\times$ and $d=(d_1, \ldots , d_n) \in \N^n$. Then:
\begin{enumerate}

\item $(1-\s)^d(\phi)=d!\l_d$ where $(1-\s)^d:=\prod_{i=1}^n(1-\s_i)^{d_i}$ and $d!=d_1!\cdots d_n!$.

\item $\sum_{\alpha \in \Pi_d}\CP_n\s^\alpha (\phi)=\CP_n$, i.e. $\bigcap_{\alpha \in \Pi_d}V(\s^\alpha (\phi))=\emptyset$.

\item For every automorphism $\tau\in \Aut_K(\CP_n)$, the automorphisms $\s_1'=\tau\s_1\tau^{-1}, \ldots , \s_n'=\tau\s_n\tau^{-1}\in \Aut_K(\CP_n)$  commute
 and $\sum_{\alpha \in \Pi_d}\CP_n\s'^\alpha (\tau(\phi))=\CP_n$, i.e. $\bigcap_{\alpha \in \Pi_d}V(\s'^\alpha (\tau(\phi)))=\emptyset$.

\end{enumerate}
\end{theorem}

\begin{proof} 1. Since $\CP_n=\CP_{n-1}\t K[H_n]$, $\phi=\phi_{n-1}H_n^{d_n}+\psi_{n-1}H_n^{d_n-1}+\cdots$ where $\phi_{n-1},\psi_{n-1},\ldots \in \CP_{n-1}$. Then
$$ (1-\s_n)^{d_n}(\phi)=d_n!\phi_{n-1}\in \CP_{n-1}$$
and the leading term of the polynomial $ d_n!\phi_{n-1}\in \CP_{n-1}$ is $d_n\l_d H_1^{d_1}\cdots H_{n-1}^{d_{n-1}}$. Now, the result follows by induction on $n$ (or by repeating the above computation $n-1$ more times).

2.  By statement 1, $K^\times \ni d!\l_d=(1-\s)^d(\phi)\in\sum_{\alpha \in \Pi_d}\CP_n\s^\alpha (\phi) $, and statement 2 follows.  

3. Clearly, the automorphisms $\s_1', \ldots , \s_n'$ commute and 
$$\CP_n=\tau (\CP_n)=\tau\bigg(\sum_{\alpha \in \Pi_d}\CP_n\s^\alpha (\phi)\bigg)=\sum_{\alpha \in \Pi_d}\CP_n\s'^\alpha (\tau(\phi)).$$
\end{proof}

\begin{corollary}\label{a27Mar24}%\marginpar{a27Mar24}

Let $\CP_n=K[H_1, \ldots , H_n]$ be a polynomial algebra over a field $K$ of characteristic zero and  $\s_i\in \Aut_K(\CP_n)$ where $\s_i (H_j)=H_j-\mu_i\d_{ij}$ for $i,j=1, \ldots , n$, $\mu_i\in K^\times$ and $\d_{ij}$ is the Kronecker delta.  Let  $\phi\in \CP_n\backslash \{ 0\}$ and $\l_d H^d$ be the  leading term of the polynomial  $\phi$ with respect to  the lexicographic ordering $H_1<\cdots <H_n$ where $\l_d\in K^\times$ and $d=(d_1, \ldots , d_n) \in \N^n$. Then:
\begin{enumerate}

\item $(1-\s)^d(\phi)=d!\mu^d\l_d$ where $(1-\s)^d:=\prod_{i=1}^n(1-\s_i)^{d_i}$, $d!=d_1!\cdots d_n!$  and $\mu^d:=\mu_1^{d_1}\cdots \mu_n^{d_n}$.

\item $\sum_{\alpha \in \Pi_d}\CP_n\s^\alpha (\phi)=\CP_n$, i.e. $\bigcap_{\alpha \in \Pi_d}V(\s^\alpha (\phi))=\emptyset$.
 
\end{enumerate}
\end{corollary}

\begin{proof} Repeat the proofs of statement 1 and 2 of Theorem \ref{27Mar24} and making an obvious adjustments.
\end{proof}

{\bf Licence.} For the purpose of open access, the author has applied a Creative Commons Attribution (CC BY) licence to any Author Accepted Manuscript version arising from this submission.

{\bf Disclosure statement.} No potential conflict of interest was reported by the author.

{\bf Data availability statement.} Data sharing not applicable – no new data generated.

\small{

School of Mathematics and Statistics

University of Sheffield

Hicks Building

Sheffield S3 7RH

UK

email: v.bavula@sheffield.ac.uk}
}

\end{document}